\documentclass[aos,authoryear,reqno]{imsart}

\usepackage[latin1]{inputenc}
\usepackage{amsmath,amssymb,amsfonts,latexsym,amsthm,mathrsfs,algorithm,enumerate}
\usepackage{graphicx,color,framed}
\usepackage{algpseudocode}
\usepackage{graphicx,color,framed,gensymb,upgreek,twoopt}
\RequirePackage{natbib}

\newcommand{\note}[2]{}

\newcommand{\mcg}[2][]{%
\ifthenelse{\equal{#1}{}}{{\mathcal{G}_{#2}}}{\mathcal{G}_{#2}^{(#1)}}%
}
\newcommand{\mcf}[2][]{%
\ifthenelse{\equal{#1}{}}{{\mathcal{F}_{#2}}}{\mathcal{F}_{#2}^{(#1)}}%
}
\newcommand{\mctf}[2][]{%
\ifthenelse{\equal{#1}{}}{\widetilde{\mathcal{F}}_{#2}}{\widetilde{\mathcal{F}}_{#2}^{(#1)}}%
}
\newcommand{\mcbf}[2][]{%
\ifthenelse{\equal{#1}{}}{\overline{\mathcal{F}}_{#2}}{\overline{\mathcal{F}}_{#2}^{(#1)}}%
}

\newcommand{\wrt}{with respect to}

\newcommand{\iid}{i.i.d.}
\newcommand{\as}{a.s.}

\newcommand{\ie}{i.e.}
\newcommand{\eg}{e.g.}

\newcommand{\eqsp}{\;}

\newcommand{\rset}{\ensuremath{\mathbb{R}}}
\newcommand{\nset}{\ensuremath{\mathbb{N}}}

\newcommand{\1}{\ensuremath{\mathbf{1}}}

\newcommand{\eqdef}{\ensuremath{\stackrel{\mathrm{def}}{=}}}

\newcommand{\rmd}{\ensuremath{\mathrm{d}}}
\newcommand{\rmi}{\ensuremath{\mathrm{i}}}
\newcommand{\rme}{\ensuremath{\mathrm{e}}}
\newcommand{\mcb}[1]{\ensuremath{\mathbb{B}_{\mathrm{b}}(#1)}}

\newcommand{\supnorm}[1]{\ensuremath{\left|#1\right|_{\infty}}}
\newcommand{\esssup}[2][]%
{\ifthenelse{\equal{#1}{}}{\left| #2 \right|_\infty}{\left| #2 \right|^2_{\infty}}}
\newcommand{\oscnorm}[2][]%
{\ifthenelse{\equal{#1}{}}{\ensuremath{\operatorname{osc}\left(#2\right)}}{\ensuremath{\operatorname{osc}^{#1}\!\left(#2\right)}}}
\newcommand{\essosc}[3][]%
{\ifthenelse{\equal{#1}{}}{\ensuremath{\operatorname{osc}_{#2}{\left(#3\right)}}}{\ensuremath{\operatorname{osc}^{#1}_{#2}\left(#3\right)}}}

\def\aux{{\scriptstyle{\mathrm{aux}}}}

\usepackage{ifthen}
\newboolean{sv_used}
\setboolean{sv_used}{true}
\usepackage{amsthm}
  \newtheorem{thm}{Theorem}
  \newtheorem{prop}[thm]{Proposition}
  
  \newtheorem{lem}[thm]{Lemma}
  
  \newtheoremstyle{assumption}
  {6pt}% space above
  {6pt}% space below
  {}%body font
  {-.7em}%indent amount
  {\bfseries}%Theorem head
  {.}%punctuation
  {0.1em}%space after theorem head
  {}
  \theoremstyle{assumption}
  \newtheorem{assum}{A}

  \theoremstyle{definition}
  \theoremstyle{definition}
\theoremstyle{definition} 
  \newtheoremstyle{example}
  {6pt}% space above
  {6pt}% space below
  {\itshape}%body font
  {-.5em}%indent amount
  {\itshape}%Theorem head
  {}%punctuation
  {1.5em}%space after theorem head
  {}
  \theoremstyle{definition}
\newcommand{\PP}{\ensuremath{\mathbb{P}}}
\newcommand{\PE}{\ensuremath{\mathbb{E}}}
\newcommand{\CPE}[3][]
{\ifthenelse{\equal{#1}{}}{\mathbb{E}\left[\left. #2 \, \right| #3 \right]}{\mathbb{E}_{#1}\left[\left. #2 \, \right | #3 \right]}}
\newcommand{\CPP}[3][]
{\ifthenelse{\equal{#1}{}}{\mathbb{P}\left(\left. #2 \, \right| #3 \right)}{\mathbb{P}_{#1}\left(\left. #2 \, \right | #3 \right)}}

\renewcommand{\mid}{\,|\,}
\newcommand{\pscal}[2]{\left\langle #1, #2 \right\rangle}
\newcommand{\ci}[4][]%
{%
\ifthenelse{\equal{#1}{}}{\ensuremath{#2 \perp\!\!\!\perp #3 \mid #4 }}{\ensuremath{#2 \perp\!\!\!\perp #3 \mid #4 \; \: [#1]}}%
}
\newcommand{\dlim}{\ensuremath{\stackrel{\mathcal{D}}{\longrightarrow}}}
\newcommand{\plim}{\ensuremath{\stackrel{\mathrm{P}}{\longrightarrow}}}

\newcommand{\Xset}{\ensuremath{\mathsf{X}}}

\newcommand{\Xsigma}[1][]%
{%
\ifthenelse{\equal{#1}{}}{\ensuremath{\mathcal{B}(\Xset)}}{\ensuremath{\mathcal{B}(\Xset^{#1})}}
}

\newcommand{\Yset}{\ensuremath{\mathsf{Y}}}
\newcommand{\Ysigma}{\ensuremath{\mathcal{B}(\Yset)}}

\newcommand{\chunk}[4][]%
{\ifthenelse{\equal{#1}{}}{\ensuremath{{#2}_{#3:#4}}}{\ensuremath{#2^#1}_{#3:#4}}
}
\newcommand{\Q}{\ensuremath{Q}}
\newcommand{\q}{\ensuremath{q}}
\newcommand{\Xinit}{\ensuremath{\chi}}
\newcommand{\XinitIS}[2][]
{\ifthenelse{\equal{#1}{}}{\ensuremath{\rho_{#2}}}{\ensuremath{\check{\rho}_{#2}}}}

\newcommand{\filt}[2][]%
{%
\ifthenelse{\equal{#1}{}}{\ensuremath{\phi_{#2}}}{\ensuremath{\phi_{#1,#2}}}%
}
\newcommand{\pred}[3][]%
{%
\ifthenelse{\equal{#1}{}}{\ensuremath{\phi_{#2|#3}}}{\ensuremath{\phi_{#1,#2|#3}}}%
}

\newcommand{\adjfunc}[4][]
{\ifthenelse{\equal{#1}{}}{\ifthenelse{\equal{#4}{}}{\vartheta_{#2}}{\vartheta_{#2}(#4)}}
{\ifthenelse{\equal{#1}{smooth}}{\ifthenelse{\equal{#4}{}}{\tilde{\vartheta}_{#2}}{\tilde{\vartheta}_{#2}(#4)}}
{\ifthenelse{\equal{#1}{fully}}{\ifthenelse{\equal{#4}{}}{\vartheta^\star_{#2}}{\vartheta^\star_{#2}(#4)}}{\mathrm{erreur}}}}}

\newcommandtwoopt{\post}[4][][]%
{
\ifthenelse{\equal{#2}{}}
{\ifthenelse{\equal{#1}{}}{\ensuremath{\phi_{#3|#4}}}{\ensuremath{\phi_{#1,#3|#4}}}}%
{\ifthenelse{\equal{#2}{hat}}{\ifthenelse{\equal{#1}{}}{\ensuremath{\hat{\phi}_{#3|#4}}}{\ensuremath{\hat{\phi}_{#1,#3|#4}}}}
{\ifthenelse{\equal{#2}{tilde}}{\ifthenelse{\equal{#1}{}}{\ensuremath{\tilde{\phi}_{#3|#4}}}{\ensuremath{\tilde{\phi}_{#1,#3|#4}}}}
{\ifthenelse{\equal{#2}{tar}}{\ifthenelse{\equal{#1}{}}{\ensuremath{\hat{\phi}^{\mathrm{tar}}_{#3|#4}}}{\ensuremath{\hat{\phi}^{\mathrm{tar}}_{#1,#3|#4}}}}
{\ifthenelse{\equal{#1}{}}{\ensuremath{\hat{\phi}^{#2}_{#3|#4}}}{\ensuremath{\hat{\phi}^{#2}_{#1,#3|#4}}}}
}
}
}
}
\newcommand{\asymVar}[4][]{
\ifthenelse{\equal{#1}{}}{\ifthenelse{\equal{#4}{}}{\ensuremath{\Gamma_{#2|#3}}}{\ensuremath{\Gamma_{#2|#3}\left[#4\right]}}}
{\ifthenelse{\equal{#4}{}}{\ensuremath{\Gamma_{#1,#2|#3}}}{\ensuremath{\Gamma_{#1,#2|#3}\left[#4\right]}}}
}
\newcommand{\asymVarFearnhead}[4][]{
\ifthenelse{\equal{#1}{}}{\ensuremath{\Upsilon_{#2|#3}\left[#4\right]}}{\ensuremath{\Upsilon_{#1,#2|#3}\left[#4\right]}}
}
\newcommand{\asymVarDoucet}[5][]
{\ifthenelse{\equal{#1}{}}{\Delta^{#2}_{#3|#4}\left[#5\right]}{\Delta^{#2}_{#1,#3|#4}\left[#5\right]}}

\newcommand{\asymVarJoint}[4][]{
\ifthenelse{\equal{#1}{}}{\ensuremath{\tilde \Gamma_{#2|#3}\left[#4\right]}}{\ensuremath{\tilde \Gamma_{#1,#2|#3}\left[#4\right]}}
}

\newcommandtwoopt{\incrasymVarJoint}[5][][]
{
\ifthenelse{\equal{#1}{}}
    {\ifthenelse{\equal{#2}{forward}}{\ensuremath{\tilde{\sigma}^{2,\mathrm{f}}_{#3|#4}\left[#5\right]}}{\ensuremath{\tilde{\sigma}^{2,\mathrm{b}}_{#3|#4}\left[#5\right]}}}
    {\ifthenelse{\equal{#2}{forward}}{\ensuremath{\tilde{\sigma}^{2,\mathrm{f}}_{#1,#3|#4}\left[#5\right]}}{\ensuremath{\tilde{\sigma}^{2,\mathrm{b}}_{#1,#3|#4}\left[#5\right]}}}
}

\newcommand{\incrasymVar}[4][]{
\ifthenelse{\equal{#1}{}}{\ensuremath{\sigma^2_{#2|#3}\left[#4\right]}}{\ensuremath{\sigma^2_{#1,#2|#3}\left[#4\right]}}
}

\newcommand{\logl}[2][]%
{%
\ifthenelse{\equal{#1}{}}{\ensuremath{\ell_{#2}}}{\ensuremath{\ell_{#1,#2}}}%
}
\newcommand{\lhood}[2][]%
{%
\ifthenelse{\equal{#1}{}}{\ensuremath{\mathrm{L}_{#2}}}{\ensuremath{\mathrm{L}_{#1,#2}}}%
}
\newcommand{\cc}[2][]%
{%
\ifthenelse{\equal{#1}{}}{\ensuremath{c_{#2}}}{\ensuremath{c_{#1,#2}}}%
}
\newcommand{\forvar}[2][]%
{%
\ifthenelse{\equal{#1}{}}{\ensuremath{\alpha_{#2}}}{\ensuremath{\alpha_{#1,#2}}}%
}
\newcommand{\nforvar}[2][]%
{%
\ifthenelse{\equal{#1}{}}{\ensuremath{\bar{\alpha}_{#2}}}{\ensuremath{\bar{\alpha}_{#1,#2}}}%
}

\newcommandtwoopt{\BK}[3][][]%
{%
\ifthenelse{\equal{#2}{}}{\ifthenelse{\equal{#1}{}}{\ensuremath{\mathrm{B}_{#3}}}{\ensuremath{\mathrm{B}_{#1,#3}}}}%
{\ifthenelse{\equal{#1}{}}{\ensuremath{\hat{\mathrm{B}}_{#3}}}{\ensuremath{\hat{\mathrm{B}}_{#1,#3}}}}
}

\newcommand{\filtfunc}[2][]%
{%
\ifthenelse{\equal{#1}{}}{\ensuremath{\tau_{#2}}}{\ensuremath{\tau_{#1,#2}}}%
}
\newcommand{\NISE}[4][]%
{%
\ifthenelse{\equal{#1}{}}{\ensuremath{\tilde{#2}^{\scriptstyle \mathrm{IS}}_{#4}\left(#3 \right)}}{\ensuremath{\tilde{#2}^{\scriptstyle \mathrm{IS}}_{#1,#4}\left(#3 \right)}}%
}
\newcommand{\ISE}[4][]%
{%
\ifthenelse{\equal{#1}{}}{\ensuremath{\widehat{#2}^{\scriptstyle  \mathrm{IS}}_{#4} \left(#3 \right)}}{\ensuremath{\widehat{#2}^{\scriptstyle  \mathrm{IS}}_{#1,#4} \left(#3 \right)}}%
}
\newcommand{\SIRE}[4][]%
{%
\ifthenelse{\equal{#1}{}}{\ensuremath{\hat{#2}^{\scriptstyle  \mathrm{SIR}}_{#4} \left(#3 \right)}}{\ensuremath{\hat{#2}^{\scriptstyle  \mathrm{SIR}}_{#1,#4} \left(#3 \right)}}%
}
\newcommand{\MCE}[3]%
{
{\ensuremath{\hat{#1}^{\scriptstyle  \mathrm{MC}}_{#3} \left(#2 \right)}}%
}
\newcommand{\kiss}[3][]
{\ifthenelse{\equal{#1}{}}{p_{#2}}
{\ifthenelse{\equal{#1}{fully}}{p^{\star}_{#2}}
{\ifthenelse{\equal{#1}{smooth}}{\tilde{r}_{#2}}{\mathrm{erreur}}}}}

\newcommand{\epart}[2]{\ensuremath{\xi_{#1}^{#2}}}

\newcommand{\epartpred}[2]{\ensuremath{\xi_{#1}^{#2}}}
\newcommand{\epartpredIndex}[2]{\ensuremath{\xi_{#1}^{#2_{#1}}}}

\newcommand{\ewght}[2]{\ensuremath{\omega_{#1}^{#2}}}
\newcommand{\ewghtfunc}[1]{\ensuremath{\omega_{#1}}}
\newcommand{\swghtIndex}[3]{\ensuremath{\omega_{#1|#2}^{#3_{#1:#2}}}}

\newcommand{\swght}[2]{\ensuremath{\omega_{#1}^{#2}}}

 \newcommand{\swghtcIndex}[2]{\ensuremath{\varpi_{#1}^{#2_{#1}}}}

\newcommand{\sumwght}[2][]{%
\ifthenelse{\equal{#1}{}}{\ensuremath{\Omega_{#2}}}{\ensuremath{\Omega_{#2}^{(#1)}}}}
\newcommand{\tsumweight}[2][]{%
\ifthenelse{\equal{#1}{}}{\ensuremath{\widetilde{\Omega}_{#2}}}{\ensuremath{\widetilde{\Omega}_{#2}^{(#1)}}}}
\newcommand{\indexresidual}[2][]{%
\ifthenelse{\equal{#1}{}}{\ensuremath{J_{#2}}}{\ensuremath{J_{#2}^{(#1)}}}}

\newcommand{\efilt}[2][]%
{%
\ifthenelse{\equal{#1}{}}{\ensuremath{\hat{\phi}_{#2}}}{\ensuremath{\hat{\phi}_{#1,#2}}}%
}
\newcommand{\epost}[3][]%
{%
\ifthenelse{\equal{#1}{}}{\ensuremath{\hat{\phi}_{#2|#3}}}{\ensuremath{\hat{\phi}_{#1,#2|#3}}}%
}

\newcommandtwoopt{\backDist}[4][][]%
{
\ifthenelse{\equal{#1}{}}{\ifthenelse{\equal{#2}{}}{\psi_{#3|#4}}{\psi_{#2,#3|#4}}}%
{\ifthenelse{\equal{#1}{hat}}{\ifthenelse{\equal{#2}{}}{\hat{\psi}_{#3|#4}}{\hat{\psi}_{#2,#3|#4}}}
{\ifthenelse{\equal{#1}{tar}}{\ifthenelse{\equal{#2}{}}{\hat{\psi}^{\mathrm{tar}}_{#3|#4}}{\hat{\psi}^{\mathrm{tar}}_{#2,#3|#4}}}
{\ifthenelse{\equal{#1}{aux}}{\ifthenelse{\equal{#2}{}}{\hat{\psi}^{\mathrm{aux}}_{#3|#4}}{\hat{\psi}^{\mathrm{aux}}_{#2,#3|#4}}}
}
}
}
}

\newcommand{\backsumwght}[3][]{%
\ifthenelse{\equal{#1}{}}{\ensuremath{{\check{\Omega}}_{#2|#3}}}{\ensuremath{{\check{\Omega}}_{#2|#3}^{(#1)}}}}
\newcommand{\backasymVar}[4][]{
\ifthenelse{\equal{#1}{}}{\ensuremath{{\check{\Gamma}}_{#2|#3}\left[#4\right]}}{\ensuremath{\check{\Gamma}_{#1,#2|#3}\left[#4\right]}}
}

\newcommand{\instrpostaux}[2]{\ensuremath{\pi_{#1|#2}}}

\begin{document}
\begin{frontmatter}
\title{On the Forward Filtering Backward Smoothing  particle approximations of the smoothing distribution in general state spaces models}
\runauthor{Douc et al.}
\runtitle{Particle approximation of smoothing distributions}
\begin{aug}
\author{Randal Douc \corref{} \ead[label=a1]{randal.douc@it-sudparis.eu}}
\affiliation{Institut Télécom/Télécom SudParis}
\author{Aurélien Garivier, \ead[label=a2]{aurelien.garivier@telecom-paristech.fr}}
\affiliation{CNRS/Télécom ParisTech, CNRS UMR 5141}
\author{Eric Moulines, \ead[label=a2]{eric.moulines@telecom-paristech.fr}}
\affiliation{Institut Télécom/Télécom ParisTech, CNRS UMR 5141}
\author{Jimmy Olsson \ead[label=a3]{jimmy@maths.lth.se}}
\affiliation{Center of Mathematical Sciences}
\address{Télécom SudParis, \\ 9 rue Charles Fourier, \\ 91011 Evry-Cedex, France \\ \printead{a1}}
\address{Télécom ParisTech, \\ 46 rue Barrault, \\ 75634 Paris, Cédex 13, France \\ \printead{a2}}
\address{Center of Mathematical Sciences, \\ Lund University, \\ Box 118 SE-22100 Lund, SWEDEN \\ \printead{a3}}
\end{aug}

\begin{abstract}
A prevalent problem in general state-space models is the approximation of the smoothing distribution of a state, or a sequence of states,
conditional on the observations from the past, the present, and the future. The aim of this paper is to provide a rigorous foundation for the calculation, or approximation, of
such smoothed distributions, and to analyse in a common unifying framework different schemes to reach this goal. Through a cohesive and generic exposition of the
scientific literature we offer several novel extensions allowing to approximate joint smoothing distribution in the most general case with a cost growing linearly with the
number of particles.
\end{abstract}


\end{frontmatter}
\newpage

\section{Introduction}

Consider the nonlinear \emph{state space model}, where the \emph{state process} $\{ X_t \}_{t \geq 0}$ is a Markov chain on some general state space $(\Xset, \Xsigma)$ having initial distribution $\Xinit$ and transition kernel $Q$. The state process is hidden but partially observed through the \emph{observations} $\{ Y_t \}_{t \geq 0}$, which are $\Yset$-valued random variables being independent conditionally on the latent state sequence $\{ X_t \}_{t \geq 0}$; in addition, there exists a $\sigma$-finite measure $\lambda$ on $(\Yset, \Ysigma)$, and a transition density function $x \mapsto g(x, y)$, referred to as the \emph{likelihood}, such that $\CPP{Y_t \in A}{X_t} = \int_A g(X_t, y) \, \lambda(dy)$ for all $A \in \Ysigma$. The kernel $Q$ and the likelihood function $x \mapsto g(x, y)$ are assumed to be known. We shall consider the case in which the observations have arbitrary but fixed values $\chunk{y}{0}{T} \eqdef [y_0,\dots,y_T]$.

Statistical inference in general state space models involves computing the \emph{posterior distribution} of a batch of state variables $\chunk{X}{s}{s'}$ conditioned on a batch of observations $\chunk{Y}{t}{T}$, which we denote by $\post{s:s'}{t:T}$ (the dependence on the observations $\chunk{Y}{t}{T}$ is implicit). The posterior distribution can be computed in closed form only in very specific cases, principally, when the state space model is linear and Gaussian or when the state space $\Xset$ is a finite set. In the vast majority of cases, nonlinearity or non-Gaussianity render analytic solutions intractable \citep{anderson:moore:1979,kailath:sayed:hassibi:2000,ristic:arulampalam:gordon:2004,cappe:moulines:ryden:2005}.

These limitations have stimulated the interest in alternative strategies being able to handle more general state and measurement equations without putting strong a priori constraints on the behaviour of the posterior distributions. Among these, \emph{Sequential Monte Carlo} (SMC) \emph{methods} play a central role. SMC methods refer to a class of algorithms for approximating a \emph{sequence of probability distributions} over a \emph{sequence of probability spaces} by updating recursively a set of random \emph{particles} with associated nonnegative \emph{weights}. These algorithms can be seen as a combination of the \emph{sequential importance sampling} and \emph{sampling importance resampling} methods introduced in \cite{handschin:mayne:1969} and \cite{rubin:1987}, respectively. SMC methods have emerged as a key tool for approximating state posterior distributions in general state space models; see, for instance, \citep{liu:chen:1998,liu:2001,doucet:defreitas:gordon:2001,ristic:arulampalam:gordon:2004} and the references therein.

The recursive formulas generating the filtering distribution $\post{T}{0:T}$ and the joint smoothing distributions
$\post{0:T}{0:T}$ are closely related. Using the  basic \emph{filtering} version of the particle filter %\eqref{eq:unnormalized-importance-weights}
actually provides as a by-product an approximation of the joint smoothing distribution in the sense that the particle \emph{paths} and their associated weights can be considered as a weighted sample approximating $\post{0:T}{0:T}$. From these joint draws one may readily obtain fixed lag or fixed interval smoothed samples by simply extracting the required components from the sampled particle paths and retaining the same weights.
This appealingly simple scheme can be used successfully for estimating the smoothing joint smoothing distribution for small values of $T$ or any marginal smoothing distribution $\post{s}{0:T}$, with $s \leq T$, when $s$ and $T$ are close; however, when $T$ is large or when $s$ and $T$ are remote, the associated particle approximations are inaccurate \cite{godsill:doucet:west:2004}.

In this article, we consider the forward filtering backward smoothing (FFBSm) algorithm and the forward filtering backward simulation (FFBSi) sampler. These algorithms share some similarities with the forward-backward algorithm for discrete state-space HMM. The FFBSm algorithm  consists in reweighting, in a backward pass, the weighted sample approximating the filtering distribution (see \cite{kitagawa:1996}, \cite{huerzeler:kuensch:1998}, \cite{doucet:godsill:andrieu:2000}). The FFBSi sample, conditionally independently to the particles and the weights obtained in the forward path, realizations of the joint smoothing fixed interval smoothing distribution; see \cite{godsill:doucet:west:2004}.

The complexity of the FFBSm algorithm to estimate the marginal fixed interval smoothing distribution or of the original formulation of the FFBSi sampler grows generally as the square of the number of particles $N$  multiplied by the time horizon $T$. This complexity can be linear in $N$ for some specific examples. Otherwise, some tricky algorithms should be developed to overcome this problem, see for example \cite{klaas:briers:defreitas:doucet:maskell:lang:2006}. Note that these computational techniques lead to algorithms with complexity of order $N \log(N)$, but this reduction in complexity comes at the price of introducing some level of approximations (truncation) which in practice introduce some bias which might be difficult to control. In this paper, a modification of the original FFBSi algorithm is presented, having a complexity which grows linearly in $N$, without having to truncate the density or to use intricate data structures.

 The FFBSm and FFBSi algorithms are very challenging to analyze and, up to now, only a consistency result is available in \cite{godsill:doucet:west:2004} (the proof of this result being plagued by an error). The FFBSm estimate and the FFBSi trajectories explicitly depend upon all the particles and weights drawn before and after this time instant. It is therefore impossible to analyze directly the
convergence of this approximation using the standard techniques developed to study the interacting particle approximations of the Feynman-Kac flows (see \cite{delmoral:2004} or \cite{douc:moulines:2008}).

The paper is organized as follows. In Section \ref{sec:FFBS}, the FFBSm algorithm and the FFBSi sampler are introduced. An exponential deviation inequality is first provided in Section~\ref{sec:exponentialFFBS} for the fixed-interval joint smoothing distribution. A Central Limit Theorem (CLT) for this quantity is then obtained in Section~\ref{sec:CLTFFBS}. Time-uniform exponential bounds are then computed for the FFBSm marginal smoothing distribution estimator, under  mixing conditions on the kernel $Q$, in Section~\ref{sec:TimeUniformExponentialFFBS}. Finally, under the same mixing condition, an explicit bound for the variance of the marginal smoothing distribution estimator is derived in Section~\ref{sec:TimeUniformCLTFFBS}.

\paragraph{Notations and Definitions}
\label{sec:Notations and Definitions}
We denote $a_{m:n} \eqdef (a_m, \ldots, a_n)$ and $[a_{m:n},b_{p:q}]\eqdef(a_m,\ldots,a_n,b_p,\ldots,b_q)$.
We assume that all random variables are defined on a common probability space $(\Omega, \mcf{}, \PP)$. A state space $\Xset$ is said to be \emph{general} if it is a Polish space and its topology is metrizable by some metric $d$ such that
$(\Xset, d)$ is a complete separable metric space. We denote by $\Xsigma$ the associated Borel $\sigma$-algebra and by $\mathbb{B}_b(\Xset)$ the set of all bounded  $\Xsigma / \mathcal{B}(\rset)$-measurable functions from $\Xset$ to $\rset$.  For any measure $\mu$ on $(\Xset,\Xsigma)$ and measurable function $f$ satisfying $\int_\Xset |f(x)| \, \mu(\rmd x) < \infty$ we set $\mu(f)= \int_\Xset f(x) \, \mu(\rmd x)$. Moreover, we say that two measures $\mu$ and $\nu$ are \emph{proportional} (written $\mu \propto \nu$) if they differ only by a normalization constant.

Let $\Xset$ and $\Yset$ be two general state spaces. A kernel $V$ from $(\Xset, \Xsigma)$ to $(\Yset, \Ysigma)$ is a map from $\Xset \times \Ysigma$ into $[0, 1]$ such that, for each $A \in \Ysigma$, $x \mapsto V (x, A)$ is a nonnegative bounded measurable function on $\Xset$ and, for each $x \in \Xset$, $A \mapsto V (x, A)$ is a measure on $\Ysigma$. The function $V(\cdot, f)$ belongs to $\mathbb{B}(\Xset)$ and we sometimes use the abridged notation $Vf$ instead of $V(\cdot, f)$. For a measure $\nu$ on $(\Xset, \Xsigma)$, we denote by $\nu V$ the measure on $(\Yset , \Ysigma)$ defined by, for any $A \in \Ysigma$, $\nu V(A) = \int_\Xset V(x,A) \, \nu(\rmd x)$.

For simplicity, we consider a \emph{fully dominated} state space models for which there exists a $\sigma$-finite measure $\nu$ on $(\Xset, \Xsigma)$ such that, for all $x \in \Xset$, $\Q(x, \cdot)$ has a transition probability density $\q(x, \cdot)$ \wrt\ $\nu$. For notational simplicity, $\nu(\rmd x)$ is sometimes replaced by $\rmd x$.

For any initial distribution $\Xinit$ on $\Xset$ and any $t \leq s \leq s'\leq T$, denote by $\post[\Xinit]{s:s'}{t:T}$ the posterior distribution of the state vector $\chunk{X}{s}{s'}$ given the observations $\chunk{Y}{t}{T}$ and knowing that $X_0 \sim \Xinit$. For all $A \in \Xsigma^{\otimes (s'-s+1)}$,  this distribution may be expressed as
\begin{multline*}
\post[\Xinit]{s:s'}{t:T}(A) = \\
\frac{\idotsint \post[\Xinit]{t}{t}(\rmd x_t) \prod_{u = t+1}^T g_{u-1}(x_{u-1}) \, Q(x_{u-1}, \rmd x_u) g_T(x_T) \1_A(\chunk{x}{s}{s'})}{\idotsint \post[\Xinit]{t}{t}(\rmd x_t) \prod_{u=t+1}^T g_{u-1}(x_{u-1}) \, Q(x_{u-1}, \rmd x_u)  g_T(x_T)} \eqsp,
\end{multline*}
with the convention $\prod_{u=s}^t=1$ if $s>t$. For simplicity, we will use the shorthand notations:
\begin{equation}\label{eq:convNot}
\begin{cases}
\post[\Xinit]{s}{t:T} \eqdef \post[\Xinit]{s:s}{t:T} \eqsp, \\
\post[\Xinit]{s:s'}{T} \eqdef \post[\Xinit]{s:s'}{0:T} \eqsp, \\
\post[\Xinit]{s}{T} \eqdef \post[\Xinit]{s:s}{0:T} \eqsp.  \\
\end{cases}
\end{equation}
In fully dominated case, the smoothing distributions $\post[\Xinit]{s:s'}{t:T}$ have densities (which we will denote similarly) with respect to the product measure $\nu^{\otimes (s' - s + 1)}$.

\section{Algorithms}
\label{sec:FFBS}
Conditionally on the observations $\chunk{Y}{0}{T}$, the state sequence $\{ X_s \}_{s \geq 0}$ is a
time-inhomogeneous Markov chain. This property remains true in the \emph{time-reversed} direction,
\ie\ given a strictly positive index $T$, initial distribution $\Xinit$,
and index $s \in \{0, \dotsc, T-1 \}$, for any $f \in \mcb{\Xset}$,
$$
\CPE[\Xinit]{f(X_s)}{\chunk{X}{s+1}{T}, Y_{0:T}} = \CPE[\Xinit]{f(X_s)}{X_{s+1}, Y_{s:T}}= \BK[\Xinit]{s}(X_{s+1}, f) \eqsp.
$$
where  $\BK[\Xinit]{s}(x_{s+1}, \cdot)$ is the  \emph{backward kernel}. In the fully dominated case, this kernel may be expressed as
\begin{equation}
\label{eq:backward-kernel}
\BK[\Xinit]{s}(x_{s+1}, A)= \int_A \frac{\post[\Xinit]{s}{s}(x)q(x, x_{s+1})}
{\int \post[\Xinit]{s}{s}(x')q(x', x_{s+1})\, \rmd x'} \, \1_A(x) \rmd x  \eqsp,
\end{equation}
Using these notations, for any integers $T > 0$, index $s  \in \{0,\dots, T-1\}$ and initial probability $\Xinit$, the joint smoothing distribution may, for all $f \in \mcb{\Xset^{T-s+1}}$, be recursively expressed as
\begin{align}
&\post[\Xinit]{s:T}{T}(f) = \CPE[\Xinit]{f( \chunk{X}{s}{T})}{\chunk{Y}{0}{T}}\nonumber \\
%& \quad = \idotsint f(\chunk{x}{s}{T}) \: \post[\Xinit]{T}{T}(\rmd x_T) \prod_{u = s}^{T-1} \bk[\Xinit]{u}(x_{u+1},x_{u}) \, \rmd x_{u} \nonumber \\
& \quad = \idotsint f(\chunk{x}{s}{T}) \: \BK[\Xinit]{s}(x_{s+1}, \rmd x_s) \, \post[\Xinit]{s+1:T}{T}(\rmd \chunk{x}{s+1}{T}) \eqsp, \label{eq:smoothing:backw_decomp}
\end{align}
with $\post[\Xinit]{T:T}{T}= \post[\Xinit]{T}{T}$ being the filtering distribution at time $T$.
If $f$ depends on the first component $x_s$ only, then Eq.~\eqref{eq:smoothing:backw_decomp} yields the marginal smoothing distribution,
which is defined recursively by:
\begin{equation}
\post[\Xinit]{s}{T}(f) = \iint f(x_s) \: \BK[\Xinit]{s}(x_{s+1}, \rmd x_s) \, \post[\Xinit]{s+1}{T}(\rmd x_{s+1}) \eqsp.
\label{eq:smoothing:marg:backw_decomp}
\end{equation}
The method proposed by \cite{huerzeler:kuensch:1998,doucet:godsill:andrieu:2000} consists in approximating the smoothing distribution by storing the particles and associated weights obtained in a forward filtering pass and revising the weights in a backward smoothing pass.
In the forward pass, particle approximations of the filtering distributions $\post[\Xinit]{s}{s}$ are computed recursively for $s = 0, \ldots, T$.
Each approximation is formed by a set of particles $\{ \epart{s}{i} \}_{i=1}^N$ and associated importance weights $\{ \ewght{s}{i} \}_{i=1}^N$ according to
\begin{equation}
\label{eq:approximation-filtering-distribution}
\post[\Xinit][hat]{s}{s}(\rmd x) = \sumwght{s}^{-1} \sum_{i=1}^N \ewght{s}{i} \delta_{\epart{s}{i}}(\rmd x) \eqsp,
\end{equation}
where $\sumwght{s}= \sum_{i=1}^N \ewght{s}{i}$ and $\delta_x$ denotes the Dirac mass located at $x$. There are several ways of producing such weighted samples $\{(\epart{s}{i}, \ewght{s}{i})\}_{i=1}^N$;
see \cite{doucet:defreitas:gordon:2001}, \cite{liu:2001}, \cite{cappe:moulines:ryden:2005}, and the references therein. Most of these algorithms can be recasted into the common unifying framework of the {\em auxiliary particle filter}.
Let $\{\epart{0}{i}\}_{i=1}^N$ be \iid\ random variables such that $\epart{0}{i} \sim \XinitIS{0}$ and set $\ewght{0}{i}=\frac{d \Xinit}{d \XinitIS{0}}(\epart{0}{i}) g_0(\epart{0}{i})$. By classical importance sampling, the weighted sample $\{ (\epart{0}{i}, \ewght{0}{i})\}_{i = 1}^N$ targets the distribution $\post[\Xinit]{0}{0}$. Assume now that the weighted sample $\{ (\epart{s-1}{i}, \ewght{s-1}{i})\}_{i = 1}^N$ targets $\post[\Xinit]{s-1}{s-1}$, \ie\ for $h \in \mcb{\Xset}$,
$ \sumwght{s-1}^{-1} \sum_{i=1}^N \ewght{s-1}{i} h(\epart{s-1}{i})$ is an estimate of $\int \post[\Xinit]{s-1}{s-1}(\rmd x) h(x)$. We may approximate $\post[\Xinit]{s}{s}$ by replacing $\post[\Xinit]{s-1}{s-1}$ in the forward filtering recursion
\begin{equation}
\label{eq:forward-filtering-recursion}
\post[\Xinit]{s}{s}(f) \propto \int \post[\Xinit]{s-1}{s-1}(\rmd x_{s-1}) q(x_{s-1}, x_s) g_s(x_s) f(x_s) \, \rmd x_s
\end{equation}
by its particle approximation $\post[\Xinit][hat]{s-1}{s-1}$, leading to the target distribution
\begin{equation}
\label{eq:target-forward-filtering}
\post[\Xinit][tar]{s}{s}(\rmd x) \propto \sum_{i=1}^N \ewght{s-1}{i} q(\epart{s-1}{i},x) g_s(x) \, \rmd x \eqsp.
\end{equation}
To avoid an $O(N^2)$ algorithm, \cite{pitt:shephard:1999} introduces an \emph{auxiliary variable} corresponding to the selected particle index and target instead the probability density
\begin{equation}
\label{eq:target-auxiliary}
\post[\Xinit][\aux]{s}{s}(i,x_s) \propto \ewght{s-1}{i} q(\epart{s-1}{i}, x_s) g_s(x_s)
\end{equation}
on the product space $\{1,\dots,N\} \times \Xset$. Since $\post[\Xinit][tar]{s}{s}$ is the marginal distribution of $\post[\Xinit][\aux]{s}{s}$ with respect to the particle index, we may sample from $\post[\Xinit][tar]{s}{s}$ by simulating instead a set $\{(I_s^i, \epart{s}{i})\}_{i=1}^N$ of indices and particle positions from an instrumental distribution having probability density
\begin{equation}
\label{eq:instrumental-distribution-filtering}
\instrpostaux{s}{s}(i,x_s) \propto \ewght{s-1}{i} \adjfunc{s}{s}{\epart{s-1}{i}} \kiss{s}{s}(\epart{s-1}{i},x_s) \eqsp,
\end{equation}
where $\{ \adjfunc{s}{s}{\epart{s-1}{i}} \}_{i = 1}^N$ are so-called \emph{adjustment multiplier weights} and $\kiss{s}{s}$ is the \emph{proposal transition} density function.
%The adjustment multiplier weights are introduced in order to increase, as the new observation becomes available,
%the weight of particles being likely to give birth to offspring particles contributing significantly to the posterior.
Each draw $(I_s^i, \epart{s}{i})$ is assigned to the weight
\begin{equation}
\label{eq:weight-update-filtering}
\ewght{s}{i} \eqdef
\frac{q(\epart{s-1}{I_s^i},\epart{s}{i}) g_s(\epart{s}{i})}{\adjfunc{s}{s}{\epart{s-1}{I_s^i}} \kiss{s}{s}(\epart{s-1}{I_s^i},\epart{s}{i})}
\eqsp,
\end{equation}
 which is proportional to $\post[\Xinit][\aux]{s}{s}(I_s^i,\epart{s}{i})/\instrpostaux{s}{s}(I_s^i,\epart{s}{i})$.
 Hereafter, the indices are discarded and $\left\{ (\epart{s}{i}, \ewght{s}{i} ) \right\}_{i=1}^N$ is taken as an
 approximation of the target distribution $\post[\Xinit]{s}{s}$.
 %The efficiency of the  algorithm depends crucially on the choice of the proposal transition and the multiplier weights;
% see for example \cite{cornebise:moulines:olsson:2008} for a discussion.
The simplest choice, yielding to the so-called \emph{bootstrap particle filter algorithm} proposed by \cite{gordon:salmond:smith:1993}, consists in setting, for all $x \in \Xset$, $\adjfunc{s}{s}{x} \equiv 1$ and $\kiss{s}{s}(x,\cdot)= q(x,\cdot)$. A more appealing choice from a theoretical standpoint --but often computationally costly-- consists in setting
$\adjfunc[fully]{s}{s}{x}= \int q(x,x_s) g_s(x_s) \, \rmd x_s$ and
\[
\kiss[fully]{s}{s}(x,x_s) = \frac{q(x,x_s) g_s(x_s)}{\adjfunc[fully]{s}{s}{x}} \eqsp.
\]
In this case, the importance weights $\{\ewght{s}{i}\}_{i=1}^N$  are all unity and the auxiliary particle filter is said to be \emph{fully adapted}. Sampling from the fully adapted version of the auxiliary particle filter is in general difficult;
the general method, based on the auxiliary accept-reject principle, proposed by \cite{huerzeler:kuensch:1998}
and \cite{kuensch:2005} for sampling from these distributions is, with few exceptions, computationally involved.
Other choices are discussed in \cite{douc:moulines:olsson:2008} and \cite{cornebise:moulines:olsson:2008}.

\subsection{The Forward Filtering Backward Smoothing algorithm}

Following \cite{godsill:doucet:west:2004}, the smoothing distribution can be approximated by filtering passes in the forward as well as the backward directions. Firstly, the particle filter is executed, while storing the weighted sample $\left\{ \left(\epart{t}{i}, \ewght{t}{i} \right)\right\}_{i=1}^N$, $1 \leq t \leq T$; secondly, starting with the particle approximation of the filtering distribution at time $T$, the importance weights are recursively updated backwards in time by combining particle estimates of the fixed interval smoothing distribution $\post[\Xinit]{s+1:T}{T}$  and the filtering distribution estimate $\post[\Xinit]{s}{s}$.
For $1 \leq s\leq t \leq T$, define $\epartpredIndex{s:t}{i} \eqdef (\epartpredIndex{s}{i}, \dots, \epartpredIndex{t}{i})$. An approximation
\begin{equation}
\label{eq:particle-approximation-stransback}
\BK[\Xinit][hat]{s}(x_{s+1}, \rmd x_s) = \sum_{i=1}^N \frac{\ewght{s}{i} q(\epart{s}{i}, x_{s+1})}{\sum_{\ell=1}^N \ewght{s}{\ell} q(\epart{s}{\ell}, x_{s+1})} \delta_{\epart{s}{i}}(\rmd x_s)
\end{equation}
of the backward kernel can be obtained by revising the weights $\{ \ewght{s}{i} \}_{i = 1}^N$ without moving the particles $\{ \epart{s}{i} \}_{i=1}^N$. If in addition the joint smoothing distribution $\post[\Xinit]{s+1:T}{T}$ is approximated at time $s + 1$ using the weighted sample $\{ (\epartpredIndex{s+1:T}{j},\swghtIndex{s+1}{T}{j}) \}$, $j_{s+1:T} \in \{1,\dots,N\}^{T-s}$, \ie\
\begin{equation}\label{eq:backApprox}
\post[\Xinit][hat]{s+1:T}{T} (\rmd x_{s+1:T}) \propto \sum_{j_{s+1:T} =1}^{N} \swghtIndex{s+1}{T}{j} \delta_{\epartpredIndex{s+1:T}{j}}(\rmd x_{s+1:T}) \eqsp,
\end{equation}
then we may substitute this and the approximation \eqref{eq:particle-approximation-stransback} of the backward kernel
 into \eqref{eq:smoothing:backw_decomp} to obtain
\[
\post[\Xinit][hat]{s:T}{T}(\rmd x_{s:T}) \propto\sum_{j_{s:T}=1}^{N} \swghtIndex{s}{T}{j}  \delta_{\epartpredIndex{s:T}{j}}(\rmd x_{s:T}) \eqsp ,
\]
where the new weight is recursively updated according to
\begin{equation}
\label{eq:FFBS-joint-weight-update}
\swghtIndex{s}{T}{j}= \frac{\ewght{s}{j_s} q(\epart{s}{j_s},\epart{s+1}{j_{s+1}})}{\sum_{\ell=1}^N \ewght{s}{\ell} q(\epart{s}{\ell},\epart{s+1}{j_{s+1}})} \swghtIndex{s+1}{T}{j}\eqsp.
\end{equation}
The estimator of the joint smoothing distribution may be rewritten as
\begin{equation}
\label{eq:forward-filtering-backward-smoothing}
\post[\Xinit][hat]{s:T}{T}(h) \eqdef \sum_{i_{s:T}=1}^N \frac{\swghtcIndex{s:T}{i} \ewght{T}{i_T}}{\sum_{j_{s:T}=1}^N  \swghtcIndex{s:T}{j} \ewght{T}{j_T}}\; h(\epartpredIndex{s:T}{i})
\end{equation}
where
\begin{equation}
\label{eq:definition-swghtcIndex}
\swghtcIndex{s:t}{i} \eqdef \prod_{u=s+1}^t \frac{\swght{u-1}{i_{u-1}}q(\epartpredIndex{u-1}{i},\epartpredIndex{u}{i})}{\sum_{\ell=1}^N \swght{u-1}{\ell}q(\epartpred{u-1}{l},\epartpredIndex{u}{i})} \eqsp,
\end{equation}
with the convention $\prod_a^b=1$ if $a>b$ so that $\swghtcIndex{s:s}{i}=1$. Since $\sum_{j_{s:T-1}} \swghtcIndex{s:t}{j}=1$, $\post[\Xinit][hat]{s:T}{T}(h)$ may be alternatively expressed as
\begin{equation}
\label{eq:forward-filtering-backward-smoothing-alternate}
\post[\Xinit][hat]{s:T}{T}(h) =  \sum_{i_{s:T}=1}^N \frac{\swghtcIndex{s:T}{i} \ewght{T}{i_T}}{\sum_{j_{T}=1}^N  \ewght{T}{j_T}}\; h(\epartpredIndex{s:T}{i}) \eqsp.
\end{equation}
This estimate of the joint smoothing distribution may be understood as an approximate importance sampling estimator.
The estimator $\post[\Xinit][hat]{s:T}{T}(h)$ is highly impractical, because  its support is  the  set of
$N^{T-s+1}$ possible particle paths $\{ \epartpredIndex{s:T}{j} \}$. Nevertheless, this estimator plays a key role in the
theoretical derivations.

The importance weight of these path particles is computed as if the path particle  $\epartpredIndex{s:T}{j}$
were simulated by drawing forward in time, for $s < t \leq T$, $\epart{t}{j_t}$ in the set  $\{ \epart{t}{i} \}_{i=1}^N$,
conditionally independently from $\{\epartpredIndex{s:t-1}{j} \}$ from the distribution
\begin{equation}
\label{eq:pseudo-proposal}
\sumwght{t-1}^{-1} \sum_{\ell=1}^N \ewght{t-1}{\ell}q(\epartpred{t-1}{\ell},\cdot) \eqsp,  \quad  i=1,\dots,N \eqsp,
\end{equation}
which approximates the predictive distribution $\post[\Xinit]{t}{t-1}$. Of course, the distribution of $\epartpredIndex{s:T}{j}$ is not exactly the product of the marginal distribution \eqref{eq:pseudo-proposal}, because the particle position are not independent -this approximation would be approximately correct for a finite block of particles selected randomly, using propagation of chaos property; see \eg\ \cite[chapter 8]{delmoral:2004} -. This is why standard results on importance sampling estimators cannot be applied to that context.

Most often, it is not required to compute the joint smoothing definition but rather the marginal smoothing distribution $\post[\Xinit]{s}{T}$ (or more generally some fixed dimensional
 marginal of the joint smoothing, $\post[\Xinit]{s:s+\Delta}{T}$ for a positive integer $\Delta$).
 Approximations of the marginal smoothing distributions may be obtained by associating to the set of particle
$\{ \epartpredIndex{s}{j} \}$, $j_{s} \in \{1,\dots,N\}^{\Delta+1}$ the weights  obtained by marginalizing
the joint smoothing weights $\{ \swghtIndex{s}{T}{j} \}$ over the components $j_{s+1:T} \in \{1,\dots,N\}^{T-s-\Delta+1}$,
$\ewght{s|T}{j_{s}}= \sum_{j_{s+1:T}=1}^N \swghtIndex{s}{T}{j}$. It is easily seen that these marginal weights can be recursively
updated as follows:
\begin{equation}
\label{eq:FFBS-marginal-weight-update}
\ewght{s|T}{i}=\sum_{j=1}^N \frac{
\ewght{s}{i}  q(\epart{s}{i},\epart{s+1}{j})}{\sum_{\ell=1}^N \ewght{s}{\ell} q(\epart{s}{\ell},\epart{s+1}{j})}\ewght{s+1|T}{j} \eqsp, \quad i=1,\dots,N
\end{equation}
The complexity of this estimator of the marginal smoothing distribution is $O(N^{2}T)$, which is manageable only if  the number of particles
is moderate. When the dimension of the input space is not too large, this computational cost can considerably reduced to $N \log(N)$,
but at the price of truncating the distribution and therefore introducing some amount of bias (see for example \cite{klaas:briers:defreitas:doucet:maskell:lang:2006}).
Note that, in certain specific scenarios (such as   discrete  Markov chains over large state space with sparse transition matrix),
the complexity can even be reduced to $O(NT)$.

\subsection{The Forward Filtering Backward Simulation}
\label{subsec:FFBSi}
Another way of understanding \eqref{eq:forward-filtering-backward-smoothing}
consists in noting that the importance weight
\eqref{eq:FFBS-joint-weight-update} is a probability distribution over $\{1,\dots,N\}^{T-s}$; more precisely,
$\swghtIndex{s+1}{T}{j}= \CPP{J_{s+1:T}= j_{s+1:T}}{\mcf{T}}$, where $\mcf{s}\eqdef \sigma\{(\epart{t}{i}, \ewght{t}{i}); 0\leq t\leq s,
1\leq i \leq N\}$ and $\{ J_u \}_{u=s}^T$ is a reversed Markov chain with a final distribution $\ewght{T}{i}/\sumwght{T}$, $i=1,\dots,N$ and
backward transition matrix $\{ \hat{B}_{i,j} \}_{i,j=1}^N$,  $\CPP{J_{s}=j}{\mcf{T} \vee \sigma(J_{s+1})} = \hat{B}_{J_s,j}$, with
\begin{equation}
\label{eq:distribution-weight}
\hat{B}_{i,j}= \frac{\ewght{s}{j} q(\epart{s}{j},\epart{s+1}{i})}{\sum_{u=1}^N \ewght{s}{u} q(\epart{s}{u},\epart{s+1}{i})} \eqsp, \quad i,j=1,\dots, N \quad \eqsp.
\end{equation}
With these definitions, the joint smoothing distribution may be written as the conditional expectation
\begin{equation}\label{eq:EspCond}
\post[\Xinit][hat]{s:T}{T}(h)= \CPE{h\left(\epartpredIndex{s:T}{J}\right)}{\mcf{T}} \eqsp.
\end{equation}
The idea of simulating the indices $J_{s:T}$ backward in time to draw approximately from the smoothing distribution $\post[\Xinit]{s:T}{T}$,
has been proposed in \cite{godsill:doucet:west:2004} (Algorithm 1, pp. 158). This algorithm proceeds recursively backward in time as follows.
At time $T$, we draw conditionally independently from  an $N$ indices $\{J^{\ell}_T \}_{\ell=1}^N$ from the distribution
$\{\ewght{T}{T}{\ell}\}_{\ell=1}^N$ (for ease of notations, we draw the same number of particles in the forward and backward passes, but there
is no need to do that).
Given now a $N$ sample $\{\chunk{J}{s+1}{T}^\ell \}_{\ell=1}^N \in \{1,\dots,N\}^{T-s}$, we draw conditionally independently
$J^{\ell}_{s} \in \{1,\dots,N\}$, $\ell=1,\dots,N$ from the distributions $\left\{\hat{B}_{j,J^{\ell}_{s+1}} \right\}_{j=1}^N$.
This algorithm is referred in the sequel to as the forward filtering backward simulation algorithm (FFBSi). This sample
yields to the following (practical) estimator of the joint fixed-interval smoothing distribution:
\begin{equation}
\label{eq:FFBSi:estimator}
\post[\Xinit][tilde]{s:T}{T}(h) = N^{-1} \sum_{\ell=1}^N h \left(\epart{s:T}{\chunk{J}{s}{T}^\ell}\right) \eqsp, \quad h \in \mcb{\Xset^{T-s+1}} \eqsp.
\end{equation}
The computational complexity for each individual realization is $O(N)$ at each time step, so the overall computational effort to
estimate $\post[\Xinit][tilde]{0:T}{T} $ is therefore $O(N^2T)$. Using the methods introduced by \cite{klaas:briers:defreitas:doucet:maskell:lang:2006},
this complexity can be further reduced to $O(N \log(N) T)$, but here again at the price of some additional approximations.
It is easy to modify this algorithm to make it linear in $N$. Assume that the transition kernel $q$ is bounded, $q(x,x') \leq \esssup{q}$.
Since $\ewght{s}{j} q(\epart{s}{j},\epart{s+1}{i}) \leq \esssup{q} \ewght{s}{j}$, for any $i,\ j \in \{1, \dots, N\}$, we may sample \eqref{eq:distribution-weight} using the accept-reject mechanism. For any $\ell=1,\dots,N$, we sample independently
indices $I^{\ell,u}_s$, $u=1,2,\dots$  from the  distribution $\{ \sumwght{s}^{-1} \ewght{s}{j} \}_{j=1}^N$  and uniform random variables $U_s^{\ell,u}$
on $[0,1]$ and let $J_s^{\ell}= I_s^{\ell,\tau_{s}^{\ell}}$  where $\tau_s^{\ell}$ is the first index $u$ for which $U_{s}^{\ell,u} \leq q(\epart{s}{I_s^{\ell,u}},\epart{s+1}{J_{s+1}^\ell})/\esssup{q}$.
%The average number of simulations required to sample $J_s^\ell$ is
%$\esssup{q} \sumwght{s}/ \sum_{u=1}^N \ewght{s}{u} q(\epart{s}{u},\epart{s+1}{J_{s+1}^{\ell}})$.
%We will show below that this quantity converges, when $N$ is large, to $\esssup{q} / \post[\Xinit]{s+1}{s}(\epart{s+1}{J_{s+1}^\ell})$,
%where the denominator is the predictive distribution of $\epart{s+1}{J_{s+1}^\ell}$. 
The complexity of the resulting algorithm is linear instead of quadratic in $N$.

\section{Exponential deviation inequality}

\label{sec:exponentialFFBS}
In this section, we establish the properties of the forward filtering backward smoothing algorithm. We first establish a non asymptotic deviation inequality.
For any function $f: \Xset^d \to \rset$, we define $\esssup{f}= \sup_{x\in\Xset^d} |f(x)|$ and $\oscnorm{f}= \sup_{(x,x') \in \Xset^d \times \Xset^d} |f(x)-f(x')|$. Denote $\bar{\nset}= \nset \cup \{\infty\}$ and consider the following assumptions. We denote by $T$ the horizon, which can be either a finite integer or infinite.
\begin{assum}\label{assum:bound-likelihood}
\ $\sup_{0 \leq t \leq T} \esssup{g_t} <\infty$.
\end{assum}
Define for $t \geq 0$ the importance weight functions:
\begin{equation}
\label{eq:definition-weightfunction-forward}
\ewghtfunc{0}(x)= \frac{d \Xinit}{d \XinitIS{0}}(x) g_0(x) \quad \text{and} \quad \ewghtfunc{t}(x,x') \eqdef \frac{q(x,x') g_t(x')}{\adjfunc{t}{t}{x} \kiss{t}{t}(x,x')} \eqsp, t \geq 1 \eqsp.
\end{equation}
\begin{assum}
\label{assum:borne-FFBS}
\ $\sup_{0 \leq t \leq T} \esssup{\adjfunc{t}{t}{}}  < \infty$ and
$\sup_{0 \leq t \leq T} \esssup{\ewghtfunc{t}}  < \infty$.
\end{assum}
The latter assumption is rather mild. It holds in particular under (A\ref{assum:bound-likelihood}) for the bootstrap filter ($q=\kiss{t}{t}$ and $\adjfunc{t}{t}{}\equiv 1$). It automatically holds in the fully adapted case ($\ewghtfunc{t} \equiv 1$).

The first step in our proof consists in obtaining an exponential deviation inequality of Hoeffding type for the
auxiliary particle approximations of the forward filtering distribution $\post[\Xinit]{t}{t}$.
Such results can be adapted from \cite[Chapter 7]{delmoral:2004}, using the Feynman-Kac representation of the auxiliary filter.
For the sake of completeness, we prove these results explicitly.
By convention, we set $\post[\Xinit][]{0}{-1}= \Xinit$ and $\adjfunc{0}{0}{} \equiv 1$.
\begin{prop}
\label{prop:exponential-inequality-forward}
Assume that A\ref{assum:bound-likelihood}--\ref{assum:borne-FFBS}. Then, for all $t \in \{0,\dots,T\}$,
there exist $0<B,\ C<\infty$ such that for all $N$, $\epsilon > 0$, and all measurable functions $h$,
\begin{equation}
\label{eq:Hoeffding-Filtering-unnormalized}
\PP\left[ \left| N^{-1}\sum_{i=1}^N \swght{t}{i} h(\epartpred{t}{i}) - \frac{\post[\Xinit]{t}{t-1}\left(g_t h\right)}{\post[\Xinit]{t-1}{t-1}(\adjfunc{t}{t}{})}\right| \geq \epsilon \right]
\leq B \rme^{-C N \epsilon^2 / \esssup{h}^2} \eqsp,
\end{equation}
\begin{equation}
\PP\left[ \left| \post[\Xinit][hat]{t}{t}(h) - \post[\Xinit]{t}{t}(h) \right| \geq \epsilon \right] \leq B \rme^{-C N \epsilon^2 / \oscnorm[2]{h}} \eqsp, \label{eq:Hoeffding-Filtering-normalized}
\end{equation}
where the weighted sample $\{ (\epart{t}{i},\ewght{t}{i}) \}_{i=1}^N$ is defined in \eqref{eq:weight-update-filtering}.
\end{prop}

\begin{proof}
We prove \eqref{eq:Hoeffding-Filtering-unnormalized} and \eqref{eq:Hoeffding-Filtering-normalized} together by induction on $t\geq 0$.
 First note that, by construction, $\{ (\epart{t}{i}, \ewght{t}{i}) \}_{1\leq i \leq N}$ are i.i.d. conditionally to the $\sigma$-field $\mcf{t-1}\eqdef \sigma\{(\epart{s}{i}, \ewght{s}{i}); 0\leq s\leq t-1, 1\leq i \leq N\}$.  Under (A\ref{assum:borne-FFBS}), we may therefore apply the  Hoeffding inequality, which implies,
\begin{equation}\label{eq:relat0}
\PP\left[\left| N^{-1} \sum_{i=1}^N \swght{t}{i}h(\epartpred{t}{i}) -\CPE{N^{-1} \sum_{i=1}^N \swght{t}{i}h(\epartpred{t}{i})}{\mcf{t-1}} \right| > \epsilon \right] \leq 2 \rme^{-2 N \epsilon^2 / (\esssup[2]{\ewghtfunc{t}} \esssup{h}^2)} \eqsp.
\end{equation}
For $t=0$,
\begin{equation*}
\CPE{N^{-1}\sum_{i=1}^N \ewght{t}{i} h(\epart{t}{i})}{\mcf{t-1}}  = \CPE{\ewght{t}{1} h(\epart{t}{1})}{\mcf{t-1}} =  \Xinit (g_0 h) = \post[\Xinit][]{0}{-1}(g_0 h) \eqsp.
\end{equation*}
Thus, \eqref{eq:Hoeffding-Filtering-normalized} follows by Lemma \ref{lem:inegEssentielle} applied with $a_N= N^{-1} \sum_{i=1}^N \swght{0}{i} h(\epartpred{0}{i})$, $b_N= N^{-1} \sum_{i=1}^N \swght{0}{i}$, $c_N=0$ and
$b=\beta=\Xinit(g_0)$, conditions \eqref{item:inegEssentielle-borne}, \eqref{item:inegEssentielle-expo1} and \eqref{item:inegEssentielle-expo2} being obviously satisfied. For $t\geq 1$, we prove \eqref{eq:Hoeffding-Filtering-unnormalized} by deriving an exponential inequality for $\CPE{N^{-1}\sum_{i=1}^N \ewght{t}{i} h(\epart{t}{i})}{\mcf{t-1}}$ thanks to the induction assumption.
It follows from the definition that
\begin{align}
\nonumber
&\CPE{N^{-1}\sum_{i=1}^N \ewght{t}{i} h(\epart{t}{i})}{\mcf{t-1}}  = \CPE{\ewght{t}{1} h(\epart{t}{1})}{\mcf{t-1}} \\
\nonumber
& \quad =  \sum_{i=1}^N \int
\frac{\ewght{t-1}{i} \adjfunc{t}{t}{\epart{t-1}{i}} \kiss{t}{t}(\epart{t-1}{i},x)}{\sum_{\ell=1}^N \ewght{t-1}{\ell} \adjfunc{t}{t}{\epart{t-1}{\ell}}}
\frac{q(\epart{t-1}{i},x) g_t(x)}{\adjfunc{t}{t}{\epart{t-1}{i}} \kiss{t}{t}(\epart{t-1}{i},x)}
 h(x) \rmd x \eqsp, \\
\label{eq:relat1}
& \quad =
\frac{  \sum_{i=1}^N \ewght{t-1}{i} \int Q(\epart{t-1}{i},\rmd x) g_t(x) h(x) }{\sum_{\ell=1}^N \ewght{t-1}{\ell} \adjfunc{t}{t}{\epart{t-1}{\ell}}}
\end{align}
We apply Lemma \ref{lem:inegEssentielle} by successively checking conditions \eqref{item:inegEssentielle-borne}, \eqref{item:inegEssentielle-expo1} and \eqref{item:inegEssentielle-expo2} with
\begin{align*}
\begin{cases}
 a_N \eqdef \sumwght{t}^{-1}\sum_{i=1}^N \ewght{t-1}{i} \int Q(\epart{t-1}{i},\rmd x) g_t(x) h(x) \\
 b_N \eqdef \sumwght{t}^{-1} \sum_{\ell=1}^N \ewght{t-1}{\ell} \adjfunc{t}{t}{\epart{t-1}{\ell}} \\
 c_N \eqdef \post[\Xinit]{t-1}{t-1}\left[\int Q(\cdot,\rmd x)g_t(x) h(x)\right] = \post[\Xinit]{t}{t-1}\left(g_t h\right)\\
 d_N \eqdef  b \eqdef \beta \eqdef \post[\Xinit]{t-1}{t-1}(\adjfunc{t}{t}{})\\
\end{cases}
\end{align*}
We have that
\begin{align*}
& \left|\frac{a_N}{b_N} \right|= |\CPE{\ewght{t}{1} h(\epart{t}{1})}{\mcf{t-1}}|\leq \esssup{\ewghtfunc{t}} \; \esssup{h} \\
& \left|\frac{c_N}{d_N} \right|= \left|\left. \post[\Xinit]{t-1}{t-1}\left[\adjfunc{t}{t}{\cdot}\int \frac{q(\cdot, x)g_t(x)}{\adjfunc{t}{t}{\cdot}\kiss{t}{t}(\cdot,x)} \kiss{t}{t}(\cdot,x) h(x)\ \rmd x\right] \right/ \post[\Xinit]{t-1}{t-1}(\adjfunc{t}{t}{}) \right| \\
& \quad \leq \esssup{\ewghtfunc{t}}\; \esssup{h} \eqsp.
\end{align*}
Thus, condition \eqref{item:inegEssentielle-borne} is satisfied.
Now, assume that the induction assumption \eqref{eq:Hoeffding-Filtering-normalized} holds where $t$ is replaced by $t-1$. Then,
$$
a_N -\frac{c_N}{d_N}b_N = \sumwght{t}^{-1}\sum_{i=1}^N \ewght{t-1}{i} H(\epart{t-1}{i})
$$
with $H(\epart{t-1}{i}) \eqdef  \int Q(\epart{t-1}{i},\rmd x) g_t(x) h(x)-\frac{\post[\Xinit]{t-1}{t-1}\left[\int Q(\cdot,\rmd x)g_t(x) h(x)\right]}{\post[\Xinit]{t-1}{t-1}(\adjfunc{t}{t}{})}\adjfunc{t}{t}{\epart{t-1}{i}}$.
And by noting that $\post[\Xinit]{t-1}{t-1}(H)=0$,
exponential inequalities for $a_N-(c_N/d_N)b_N$ and $b_N-b$ are then directly derived from the induction assumption under (A\ref{assum:bound-likelihood}-\ref{assum:borne-FFBS}). Thus Lemma \ref{lem:inegEssentielle} applies and finally \eqref{eq:Hoeffding-Filtering-unnormalized} is proved for $t\geq 1$.

To conclude, it remains to see why \eqref{eq:Hoeffding-Filtering-unnormalized} implies \eqref{eq:Hoeffding-Filtering-normalized}. Without loss of generality, we assume that $\post[\Xinit]{t}{t}(h)=0$ holds. An exponential inequality for $\sumwght{t}^{-1} \sum_{i=1}^N \swght{t}{i} h(\epartpred{t}{i})$ is obtained by applying Lemma 1 with
$$
\begin{cases}
a_N\eqdef N^{-1} \sum_{i=1}^N \swght{t}{i} h(\epartpred{t}{i}) \\
b_N\eqdef N^{-1} \sum_{i=1}^N \swght{t}{i}\\
c_N\eqdef 0\\
b \eqdef \beta \eqdef \frac{\post[\Xinit]{t-1}{t-1}\left[\int Q(\cdot,\rmd x)g_t(x)\right]}{\post[\Xinit]{t-1}{t-1}(\adjfunc{t}{t}{})}
\end{cases}
$$
where conditions \eqref{item:inegEssentielle-borne}, \eqref{item:inegEssentielle-expo1}, \eqref{item:inegEssentielle-expo2} are obviously derived from the \eqref{eq:Hoeffding-Filtering-unnormalized} since the condition $\post[\Xinit]{t}{t}(h)=0$ directly implies that
$$
\post[\Xinit]{t-1}{t-1}\left[\int Q(\cdot,\rmd x)g_t(x)h(x)\right]=0
$$
\end{proof}

Using the exponential deviation inequality for the auxiliary particle approximation of the filtering density, it is now possible to derive an exponential inequality for the
forward filtering backward smoothing approximation of the joint smoothing distribution.

\begin{thm}
\label{thm:Hoeffding-FFBS}
Assume A\ref{assum:bound-likelihood}--\ref{assum:borne-FFBS}.
Let $1\leq s \leq T$. There exist $0<B,\ C<\infty$ such that for all $N$, $\epsilon > 0$, and all measurable functions $h$,
\begin{align}
\label{eq:Hoeffding-1}
&\PP\left[ \left| \post[\Xinit][hat]{s:T}{T}(h) -  \post[\Xinit]{s:T}{T}(h) \right| \geq \epsilon \right] \leq B \rme^{-C N \epsilon^2 / \oscnorm[2]{h}} \eqsp, \\
\label{eq:Hoeffding-2}
&\PP\left[ \left| \post[\Xinit][tilde]{s:T}{T}(h) -  \post[\Xinit]{s:T}{T}(h) \right| \geq \epsilon \right] \leq B \rme^{-C N \epsilon^2 / \oscnorm[2]{h}} \eqsp.
\end{align}
where $\post[\Xinit][hat]{s:T}{T}(h)$ and $\post[\Xinit][tilde]{s:T}{T}(h)$ are defined in \eqref{eq:forward-filtering-backward-smoothing} and \eqref{eq:FFBSi:estimator}.
\end{thm}
\begin{proof}
Using \eqref{eq:EspCond} and the definition of  $\post[\Xinit][hat]{s:T}{T}(h)$, we may write
$$
\post[\Xinit][tilde]{s:T}{T}(h) -  \post[\Xinit][hat]{s:T}{T}(h)=  N^{-1} \sum_{\ell=1}^N \left[ h \left(\epart{s:T}{\chunk{J}{s}{T}^\ell}\right) - \CPE{h\left(\epartpredIndex{s:T}{J}\right)}{\mcf{T}}  \right] \eqsp,
$$
which implies \eqref{eq:Hoeffding-2} by the Hoeffding inequality and \eqref{eq:Hoeffding-1}. We now prove \eqref{eq:Hoeffding-1}.
Let $s\leq t <T$. For $h$ a measurable function defined on $\Xset^{T-s+1}$, define the kernels $L_{s,T,T}(\epartpredIndex{s:T}{i},h)\eqdef h(\epartpredIndex{s:T}{i})$ and, for $s \leq t < T$,
\begin{multline}
\label{eq:definition-Lt}
L_{s,t,T}(\epartpredIndex{s:t}{i},h) \eqdef \idotsint Q(\epartpredIndex{t}{i},\rmd x_{t+1})g_{t+1}(x_{t+1}) \times \\
 \prod_{u=t+1}^{T-1} Q(x_u,\rmd x_{u+1}) g_{u+1}(x_{u+1})  h([\epartpredIndex{s:t}{i},x_{t+1:T}])\eqsp.
\end{multline}
By construction, $L_{s,t,T}$ can be obtained recursively \emph{backwards} in time as follows:
\begin{equation}\label{eq:recursionLt}
L_{s,t-1,T}(\epartpredIndex{s:t-1}{i},h)= \int Q(\epartpredIndex{t-1}{i},\rmd x)g_{t}(x) L_{s,t,T}([\epartpredIndex{s:t-1}{i},x],h)
\end{equation}
Denote by $\1$ the function identically equal to one. Under (A\ref{assum:bound-likelihood}), $\esssup{L_{s,t,T}(\cdot,h)} \leq \esssup{L_{s,t,T}(\cdot,\1)} \esssup{h}$, and $\esssup{L_{s,t,T}(\cdot,\1)} < \infty$.
Denote
\begin{equation}
\label{eq:definition-Fs-As}
F_{s,s,T}(\epart{}{},h) \eqdef  L_{s,s,T}(\epart{}{},h) \eqsp,
\end{equation}
and for all $t \in \{ s+1, \dots, T\}$, set
\begin{equation}
\label{eq:definition-Ft}
F_{s,t,T}(\epart{}{},h) \eqdef  \sum_{i_{s:t-1}=1}^N \swghtcIndex{s:t-1}{i}
\frac{\ewght{t-1}{i_{t-1}} q(\epart{t-1}{i_{t-1}},\epart{}{})}{\sum_{\ell=1}^N \ewght{t-1}{\ell} q(\epart{t-1}{\ell},\epart{}{})} L_{s,t,T}([\epartpredIndex{s:t-1}{i},\epart{}{}],h) \eqsp,
\end{equation}
where $\swghtcIndex{s:t-1}{i}$ is defined in \eqref{eq:definition-swghtcIndex}. Define, for
$t \in \{s, \dots, T\}$,
\begin{equation}
\label{eq:definition-Ast}
A_{s,t,T}(h) \eqdef \sum_{\ell=1}^N  \ewght{t}{\ell} F_{s,t,T}(\epart{t}{\ell},h)  \eqsp.
\end{equation}
Without loss of generality, we assume that $\post[\Xinit]{s:T}{T}(h)=0$.
With the notations introduced above, $\post[\Xinit][hat]{s:T}{T}(h)$ may be expressed as the sum
\begin{align} \label{eq:decomp_Smooth}
\post[\Xinit][hat]{s:T}{T}(h)= \frac{A_{s,T,T}(h)}{A_{s,T,T}(\1)} = \frac{A_{s,s,T}(h)}{A_{s,s,T}(\1)} + \sum_{t=s+1}^{T} \left\{ \frac{A_{s,t,T}(h)}{A_{s,t,T}(\1)} - \frac{A_{s,t-1,T}(h)}{A_{s,t-1,T}(\1)} \right\}  \eqsp.
\end{align}
We now compute an exponential bound for the terms appearing the RHS of the previous identity. Note that
 $$
 \frac{A_{s,s,T}(h)}{A_{s,s,T}(\1)} = \frac{\sumwght{s}^{-1} \sum_{\ell=1}^N \ewght{s}{\ell} L_{s,s,T}(\epartpred{s}{\ell},h)}{\sumwght{s}^{-1} \sum_{\ell=1}^N \ewght{s}{\ell} L_{s,s,T}(\epartpred{s}{\ell},\1)}
 $$
 We apply Lemma \ref{lem:inegEssentielle} by successively checking conditions \eqref{item:inegEssentielle-borne}, \eqref{item:inegEssentielle-expo1} and \eqref{item:inegEssentielle-expo2} with
$a_N= \sumwght{s}^{-1} \sum_{\ell=1}^N \ewght{s}{\ell} L_{s,s,T}(\epartpred{s}{\ell},h)$, $b_N= \sumwght{s}^{-1} \sum_{\ell=1}^N \ewght{s}{\ell} L_{s,s,T}(\epartpred{s}{\ell},\1)$, $c_N=0$, and  $b=\beta=\post[\Xinit]{s}{s}(L_{s,s,T}(\cdot,\1))$.
It follows immediately from the definition that
$\left|\frac{a_N}{b_N} \right|\leq  \esssup{h}$.
Moreover,  $\post[\Xinit]{s:T}{T}(h)=0$ implies that $\post[\Xinit]{s}{s}(L_{s,s,T}(\cdot,h))=0$, conditions \eqref{item:inegEssentielle-expo1} and \eqref{item:inegEssentielle-expo2} are then directly derived from the exponential inequality for the auxiliary filter (see Proposition \ref{prop:exponential-inequality-forward}, Eq.~\eqref{eq:Hoeffding-Filtering-normalized}).

We now establish an exponential inequality for   $$A_{s,t,T}(h)/A_{s,t,T}(\1)-A_{s,t-1,T}(h)/A_{s,t-1,T}(\1)$$  using again Lemma \ref{lem:inegEssentielle}. We take $a_N= N^{-1}A_{s,t,T}(h)$, $b_N= N^{-1} A_{s,t,T}(\1)$, $c_N= N^{-1}A_{s,t-1,T}(h)$, $d_N= N^{-1} A_{s,t-1,T}(\1)$ and
\begin{equation}
\label{eq:definition-b-1}
b=\beta=\frac{\post[\Xinit]{t-1}{t-1}(L_{s,t-1,T}(\cdot,\1))}{\post[\Xinit]{t-1}{t-1}(\adjfunc{t}{t}{})} \eqsp.
\end{equation}
By definition, $\left|{a_N}/{b_N} \right|\leq  \esssup{h}$ and $\left|{c_N}/{d_N} \right|\leq  \esssup{h}$, showing Lemma \ref{lem:inegEssentielle}, condition \eqref{item:inegEssentielle-borne}. We now check condition \eqref{item:inegEssentielle-expo1}.
The function $\epartpredIndex{s:t}{i} \mapsto
L_{s,t,T}(\epartpredIndex{s:t}{i},\1)$ depends only on $\epartpredIndex{t}{i}$; with a slight abuse of notation, we set $L_{s,t,T}(\epartpredIndex{s:t}{i},\1)= L_{t,t,T}(\epartpredIndex{t}{i},\1)$.
It follows from the definition \eqref{eq:definition-Ft} that $F_{s,t,T}(\epart{}{},\1)= L_{t,t,T}(\epart{}{},\1)$.
Plugging this into the definition of $A_{s,t,T}(\1)$ yields:
$A_{s,t,T}(\1)=\sum_{\ell=1}^N \swght{t}{i} L_{t,t,T}(\epartpred{t}{i},\1)$. Condition \eqref{item:inegEssentielle-expo1} follows from Proposition \ref{prop:exponential-inequality-forward}-Eq.~\eqref{eq:Hoeffding-Filtering-unnormalized}. Finally, we check condition \eqref{item:inegEssentielle-expo2}. Write
$
a_N-\frac{c_N}{d_N}b_N =  N^{-1}\sum_{\ell=1}^N  \ewght{t}{\ell} G_{s,t,T}(\epart{t}{\ell},h)
$ where
\begin{equation}
\label{eq:definition-G}
  G_{s,t,T}(\epart{}{},h) \eqdef F_{s,t,T}(\epart{}{},h)- \frac{A_{s,t-1,T}(h)}{A_{s,t-1,T}(\1)}F_{s,t,T}(\epart{}{},\1) \eqsp.
\end{equation}
Since $\{(\epart{t}{\ell},\ewght{t}{\ell})\}_{\ell=1}^N$ are \iid\ conditionally to the $\sigma$-field  $\mcf{t-1}$, we have that $\{\ewght{t}{\ell} G_{s,t,T}(\epart{t}{\ell},h)\}_{\ell=1}^{N}$ are also \iid\ conditionally to $\mcf{t-1}$. That allows to apply the conditional Hoeffding's inequality to $ N^{-1}\sum_{\ell=1}^N  \ewght{t}{\ell} G_{s,t,T}(\epart{t}{\ell},h)$ provided that we have first checked that $\CPE{\eta^1}{\mcf{t-1}}=0$ and that $(\ewght{t}{\ell} G_{s,t,T}(\epart{t}{\ell},h))_{1\leq \ell\leq N}$ are bounded random variables.
By \eqref{eq:relat1}, for any bounded function $f$,
%Since $\{ (\ewght{t}{i},\epart{t}{i}) \}_{i=1}^N$ are
%independent conditionally to $\mcf{t-1}$ and have the same distribution, for any bounded function $f$,
\begin{equation*}
\CPE{\ewght{t}{1} f(\epart{t}{1})}{\mcf{t-1}} = \int \frac{\sum_{i=1}^N \ewght{t-1}{i} q(\epart{t-1}{i},x) }{\sum_{\ell=1}^N \ewght{t-1}{\ell} \adjfunc{t}{t}{\epart{t-1}{\ell}}}  g_t(x) f(x) \rmd x
\end{equation*}
Applying this relation with $f(\cdot)= F_{s,t,T}(\cdot,h)$ and using the recursion \eqref{eq:recursionLt}, the previous relation implies by direct calculation
\begin{equation}
\CPE{\ewght{t}{1} F_{s,t,T}(\epart{t}{1},h)}{\mcf{t-1}}=\frac{ \sum_{\ell =1}^N \ewght{t-1}{\ell} F_{s,t-1,T}(\epart{t-1}{\ell},h)}
{\sum_{\ell=1}^N \ewght{t-1}{\ell}\adjfunc{t}{t}{\epart{t-1}{\ell}}}\eqsp. \label{eq:espCondA}
\end{equation}
Therefore, since $A_{s,t-1,T}(h)/A_{s,t-1,T}(\1)$ is $\mcf{t-1}$-measurable, %and using that $F_{s,t-1,T}(\epart{}{},\1)= L_{t-1,t-1,T}(\epart{}{},\1)$,
\begin{multline}
\label{eq:ineq0}
\CPE{\ewght{t}{1} G_{s,t,T}(\epart{t}{1},h)}{\mcf{t-1}}\\
= \frac{\sum_{\ell=1}^N \ewght{t-1}{\ell} F_{s,t-1,T}(\epart{t-1}{\ell},h)}{\sum_{\ell=1}^N \ewght{t-1}{\ell} \adjfunc{t}{t}{\epart{t-1}{\ell}}} - \frac{A_{s,t-1,T}(h)}{A_{s,t-1,T}(\1)} \frac{\sum_{\ell=1}^N \ewght{t-1}{\ell} F_{s,t-1,T}(\epart{t-1}{1},\1)}{\sum_{\ell=1}^N \ewght{t-1}{\ell} \adjfunc{t}{t}{\epart{t-1}{\ell}}} =0 \eqsp,
\end{multline}
by definition \eqref{eq:definition-Ast}. Moreover, since $| F_{s,t,T}(x,h)| \leq F_{s,t,T}(x,\1) \esssup{h}$, we have $|A_{s,t,T}(h)/A_{s,t,T}(\1)| \leq \esssup{h}$, showing that
\begin{multline}
|\ewghtfunc{t}(x,x') G_{s,t,T}(x',h)|=\left| \ewghtfunc{t}(x,x') \left(F_{s,t,T}(x',h) - \frac{A_{s,t-1,T}(h)}{A_{s,t-1,T}(\1)} F_{s,t,T}(x',\1) \right)\right| \\
\leq 2 \esssup{\ewghtfunc{t}} \esssup{F_{s,t,T}(\cdot,\1)}\esssup{h} \leq 2 \esssup{\ewghtfunc{t}} \esssup{L_{t,t,T}(\cdot,\1)} \esssup{h} < \infty \eqsp.
\label{eq:borne-ewghtfuncxG}
\end{multline}
The Hoeffding inequality therefore implies:
\begin{multline*}
\PP\left(\left|a_N-\frac{c_N}{d_N}b_N\right| >\epsilon\right)= \PP\left(\left| N^{-1}\sum_{\ell=1}^N  \ewght{t}{\ell} G_{s,t,T}(\epart{t}{\ell},h)\right|>\epsilon\right)\\ \leq B \exp\left\{-CN \left(\frac{\epsilon}{\esssup{h}}\right)^2\right\} \eqsp,
\end{multline*}
showing condition \eqref{item:inegEssentielle-expo2} and concluding the proof.
\end{proof}

\section{Asymptotic normality}
\label{sec:CLTFFBS}
We now derive a Central Limit Theorem (CLT) for the forward-filtering backward-smoothing estimator \eqref{eq:forward-filtering-backward-smoothing}. Consider the following
assumption.
\begin{assum}\label{assum:bound-proposal-kernel}
for all $t \in \{1,\dots,T\}$ and $M > 0$, $\int \sup_{|x| \leq M} \kiss{t}{t}(x,x') \rmd x'  <\infty$.
\end{assum}

We first recall that, under assumption (A\ref{assum:bound-likelihood}-\ref{assum:borne-FFBS}) the auxiliary particle filter approximation of the filtering distribution satisfies a CLT (see for example \cite[Theorem 3.2]{douc:moulines:olsson:2008}).
For any bounded measurable function $h: \Xset \to \rset$, define the kernel
\begin{equation}
\label{eq:definition-Lst}
L_{s,t}(x,h) \eqdef L_{s,s,t}(x,\bar{h})
\end{equation}
where $\bar{h}(x_{s:t}) = h(x_s)$.
The quantity $L_{s,t}(\epart{}{},h)/L_{s,t}(\epart{}{},\1)$ may be interpreted as the conditional expectation of $h(X_t)$ given the observations up to time $t$ and $X_s$ evaluated at $X_s=\epart{}{}$. Moreover, for any distribution $\nu$,
\begin{equation}\label{eq:interpretation-L}
\post[\nu]{t}{s:t}(h)=\frac{\int \nu(\rmd x_s) L_{s,t}(x_s,h)}{\int \nu(\rmd x_s) L_{s,t}(x_s,\1) } \eqsp.
\end{equation}

\begin{prop}
\label{prop:CLT-auxiliary-particle-filter}
Assume A\ref{assum:bound-likelihood}--\ref{assum:borne-FFBS}.  Then, for all bounded measurable functions
$h: \Xset \to \rset$, and $0 \leq s \leq T$,
\begin{align}
\label{eq:clt-auxiliary-particle-filter}
\sqrt{N}\left(\post[\Xinit][hat]{s}{s}(h) -  \post[\Xinit]{s}{s}(h) \right) \dlim \mathcal{N}\left(0,\asymVar[\Xinit]{s}{s}{h-  \post[\Xinit]{s}{s}(h)}\right) \eqsp,
\end{align}
where
\begin{equation}
\label{eq:recursionasymVar-filtering}
\asymVar[\Xinit]{s}{s}{h}= \sum_{r=0}^s V_{\Xinit,r,s}[h] \eqsp,
\end{equation}
with
\begin{equation}\label{eq:defVOs}
V_{\Xinit,0,s}[h]= \frac{\rho_0\left( \ewghtfunc{0}^2(\cdot) L^2_{0,s}(\cdot,h) \right)}{\left(\post[\Xinit]{0}{-1}\left[g_0(\cdot) L_{0,s}(\cdot,\1) \right] \right)^2}
\eqsp,
\end{equation}
and for $1 \leq r \leq s$,
\begin{equation}\label{eq:defVrs}
V_{\Xinit,r,s}[h]= \frac{\post[\Xinit]{r-1}{r-1}\left[  \adjfunc{r}{r}{\cdot} \int \kiss{r}{r}(\cdot,x) \ewghtfunc{r}^2(\cdot,x) L^2_{r,s}(x,h) \rmd x\right] \post[\Xinit]{r-1}{r-1}(\adjfunc{r}{r}{})}{\left(\post[\Xinit]{r}{r-1}\left[g_r(\cdot) L_{r,s}(\cdot,\1) \right] \right)^2} \eqsp.
\end{equation}
\end{prop}
Using the CLT for the auxiliary particle approximation of the filtering distribution, we establish a CLT for the auxiliary particle
approximation of the fixed interval joint smoothing distribution. The proof is established using this time a recursion going forward in time, \ie\ a CLT for  $\post[\Xinit][hat]{s:t}{t}(\cdot)$ is deduced from a CLT for $\post[\Xinit][hat]{s:t-1}{t-1}(\cdot)$. The proof is based on the techniques developed in \cite{douc:moulines:2008} (extending \cite{chopin:2004} and \cite{kuensch:2005}) which are tailored to the analysis of sequential Monte-Carlo algorithms.
Define, for $0 \leq s < t \leq T$,  $q_{s,t-1}(x_{s:t-1},x)= q(x_{t-1},x)$
\begin{equation}
\label{eq:definition-upsilon}
\upsilon_{s,t,T}(x,h) \eqdef \adjfunc{t}{t}{x} \int \kiss{t}{t}(x,x') \ewghtfunc{t}^2(x,x') g^2_{s,t,T}(x',h) \rmd x'\eqsp,
\end{equation}
where
\begin{equation}
\label{eq:definition-g}
g_{s,t,T}(x,h) \eqdef \frac{\post[\Xinit]{s:t-1}{t-1}(q_{s:t-1}(\cdot,x) L_{s,t,T}([\cdot,x],h))}{\post[\Xinit]{t-1}{t-1}(q(\cdot,x))} \eqsp.
\end{equation}
\begin{thm}
Assume A\ref{assum:bound-likelihood}--\ref{assum:borne-FFBS}. Let $s \leq T$.
Then, for all bounded measurable functions $h: \Xset^{T-s+1} \to \rset$,
\begin{equation}
\label{eq:clt}
\sqrt{N}\left(\post[\Xinit][hat]{s:T}{T}(h) -  \post[\Xinit]{s:T}{T}(h) \right)\\\dlim \mathcal{N}\left(0,\asymVar[\Xinit]{s:T}{T}{h-  \post[\Xinit]{s:T}{T}(h)}\right) \eqsp,
\end{equation}
with
\begin{multline}
\label{eq:expression-covariance}
\asymVar[\Xinit]{s:T}{T}{h}  \eqdef \frac{\asymVar[\Xinit]{s}{s}{L_{s,s,T}(\cdot,h)}}{\post[\Xinit]{s}{s}^2[L_{s,s,T}(\cdot,\1)]} \\+
\sum_{t=s+1}^T \frac{\post[\Xinit]{t-1}{t-1}(\upsilon_{s,t,T}(\cdot,h)) \post[\Xinit]{t-1}{t-1}(\adjfunc{t}{t}{})}{\post[\Xinit]{t}{t-1}^2[g_t(\cdot) L_{t,t,T}(\cdot,\1)]}
\end{multline}
where $\asymVar[\Xinit]{s}{s}{}$ is defined in \eqref{eq:clt-auxiliary-particle-filter}.
Moreover,
\begin{multline}
\label{eq:cltSi}
\sqrt{N}\left(\post[\Xinit][tilde]{s:T}{T}(h) -  \post[\Xinit]{s:T}{T}(h) \right) \\
\dlim \mathcal{N}\left(0,\post[\Xinit]{s:T}{T}^2\left[h-  \post[\Xinit]{s:T}{T}(h)\right]+\asymVar[\Xinit]{s:T}{T}{h-  \post[\Xinit]{s:T}{T}(h)}\right) \eqsp,
\end{multline}

\end{thm}
\begin{proof}
Without loss of generality, we assume that $\post[\Xinit]{s:T}{T}(h)=0$.
Denote by $\pscal{\cdot}{\cdot}$ the scalar product, $V^N_{s,T}= [V^N_{s,s,T}, \dots, V^N_{s,T,T}]$ and $W^N_{s,T}= [W^N_{s,s,T}, \dots, W^N_{s,T,T}]$ the vectors given by
\begin{align}
\label{eq:definition-VssT}
&V_{s,s,T}^N(h)= N^{1/2} \frac{A_{s,s,T}(h)}{\sumwght{s}} \eqsp, \\
\label{eq:definition-VstT}
&V_{s,t,T}^N(h) \eqdef  N^{-1/2} \left( A_{s,t,T}(h) -\frac{A_{s,t-1,T}(h)}{A_{s,t-1,T}(\1)} A_{s,t,T}(\1) \right) \eqsp, t=s+1, \dots, T \eqsp,
\\
&W_{s,s,T}^N = \frac{\sumwght{s}}{A_{s,s,T}(\1)} \eqsp, \quad W_{s,t,T}^N = \frac{N}{A_{s,t,T}(\1)} \eqsp, t=s+1,\dots,T\eqsp.
\label{eq:definition-W}
\end{align}
where $A_{s,t,T}$ is defined in \eqref{eq:definition-Ast}.  Using these notations, we decompose $\sqrt{N} \post[\Xinit][hat]{s:T}{T}(h)$ as follows
\begin{align} \label{eq:decomp_Smooth_1}
\sqrt{N} \post[\Xinit][hat]{s:T}{T}(h)= \pscal{V^N_{s,T}(h)}{W^N_{s,T}} \eqsp,
\end{align}
Since $W_{s,s,T}^N= \left( \post[\Xinit][hat]{s}{s}(L_{s,s,T}(\cdot,\1))\right)^{-1}$ and similarly for $t=s+1,\dots,T$, $W_{s,t,T}^N= \left( N^{-1} \sum_{\ell=1}^N \ewght{t}{\ell} L_{t,t,T}(\epart{t}{\ell},\1)\right)^{-1}$,
Proposition \ref{prop:exponential-inequality-forward}, \eqref{eq:Hoeffding-Filtering-unnormalized} and \eqref{eq:Hoeffding-Filtering-normalized} show that
\begin{align}
&W_{s,s,T}^N \plim_{n \to \infty} \frac{1}{\post[\Xinit]{s}{s}(L_{s,s,T}(\cdot,\1))} \\
\label{eq:limit-WstT}
&W_{s,t,T}^N \plim_{N \to \infty} \frac{\post[\Xinit]{t-1}{t-1}(\adjfunc{t}{t}{})}{\post[\Xinit]{t}{t-1}(g_t(\cdot) L_{t,t,T}(\cdot,\1))} \eqsp.
\end{align}
Therefore \eqref{eq:clt} follows from the application of the Slutsky Lemma provided that we establish a multivariate CLT for the sequence of random vectors $V^N_{s,T}(h)$. For that purpose, we show that for any $t \in \{s, \dots, T\}$ and any scalars $\alpha_s,\dots,\alpha_t$,
\begin{equation}
\label{eq:multivariate-CLT-FFBS}
\sum_{r=s}^t \alpha_r V_{s,r,T}^N(h) \dlim_{N \to \infty} \mathcal{N}\left(0, \sum_{r=1}^t \alpha_r^2 \incrasymVar[\Xinit]{s:r}{T}{h} \right) \eqsp,
\end{equation}
where $\incrasymVar[\Xinit]{s}{T}{h}= \asymVar[\Xinit]{s}{s}{L_{s,s,T}(\cdot,h)}$, and for $r \in \{ s+1,\dots, T\}$,
\begin{equation}
\label{eq:definition-Sigma}
\incrasymVar[\Xinit]{s:r}{T}{h} \eqdef \frac{\post[\Xinit]{s:r-1}{r-1}[\upsilon_{s,r,T}(\cdot,h)]}{\post[\Xinit]{r-1}{r-1}(\adjfunc{r}{r}{})} \eqsp.
\end{equation}
 First consider $t=s$. Since $A_{s,s,T}(h)= \sum_{\ell=1}^N \ewght{s}{\ell} L_{s,s,T}(\epart{s}{\ell},h)$ and $\post[\Xinit]{s}{s}(L_{s,s,T}(\cdot,h))= 0$,  Proposition \ref{prop:CLT-auxiliary-particle-filter} shows that
 $N^{1/2} \sumwght{s}^{-1} A_{s,s,T}(h)$ is asymptotically normal with zero mean and variance $\Sigma_{s,s,T}(h)$.
 Assume now that the property holds for some $t-1 \geq s$. We apply the results on triangular array of dependent random variables
 developed in \cite{douc:moulines:2008}; $V_{s,t,T}^N(h)$ may be expressed as
\begin{equation}
\label{eq:definition-V-triangular-array}
V_{s,t,T}^N(h)= \sum_{\ell=1}^N U_{M,\ell}\eqsp,  \quad U_{M,\ell}= N^{-1/2} \ewght{t}{\ell} G_{s,t,T}(\epart{t}{\ell},h) \eqsp,
\end{equation}
where $G_{s,t,T}$ is defined in \eqref{eq:definition-G}. Eq.~\eqref{eq:ineq0} shows that $\CPE{U_{M,\ell}}{\mcf{t-1}}=0$, $\ell=1,\dots,N$,
which implies:
$$
\CPE{\sum_{r=s}^t \alpha_r V_{s,r,T}^N(h)}{\mcf{t-1}}= \sum_{r=s}^{t-1} \alpha_r V_{s,r,T}^N(h) \eqsp.
$$
The induction assumption shows that $\CPE{\sum_{r=s}^t \alpha_r V_{s,r,T}^N(h)}{\mcf{t-1}}$ converges in distribution to a centered Gaussian distribution with covariance $\sum_{r=1}^{t-1} \alpha_r^2 \Sigma_{s,r,T}(h)$. We will now prove that
\[
\CPE{\exp\left( \rmi u \sqrt{N} V^N_{s,t,T}(h) \right)}{\mcf{t-1}}
\plim_{N \to \infty} \exp\left\{-\frac{u^2}{2}  \Sigma_{s,r,T}(h) \right\} \eqsp.
\]
By \cite[Corollary 11]{douc:moulines:2008}, setting
it remains to show that
\begin{align}
\label{eq:triangular-2}
&\sum_{\ell=1}^N \CPE{U_{M,\ell}^2}{\mcg{N,\ell-1}} \plim \Sigma_{s,r,T}(h) \eqsp, \\
\label{eq:triangular-3}
&\sum_{\ell=1}^N \CPE{U_{M,\ell}^2 \1_{\{|U_{M,\ell}| \geq \epsilon\}}}{\mcg{N,\ell-1}} \plim 0 \eqsp, \text{for any $\epsilon > 0$} \eqsp,
\end{align}
where, for $\ell \in \{1,\dots,N\}$, $\mcg{N,\ell} = \mcf{t-1} \vee \sigma( \ewght{t}{j}, \epart{t}{j}, j \leq \ell)$.
We first prove \eqref{eq:triangular-2}. It follows from the definitions that
\begin{align*}
&\sum_{\ell=1}^N  \CPE{U^2_{N,\ell}}{\mcg{N,\ell-1}} =
\CPE{\left( \ewght{t}{1} G_{s,t,T}(\epart{t}{1},h)\right)^2}{\mcf{t-1}}
\\&= \int \sum_{\ell=1}^N
\frac{\ewght{t-1}{\ell} \adjfunc{t}{t}{\epart{t-1}{\ell}} \kiss{t}{t}(\epart{t-1}{\ell},x)}
{\sum_{j=1}^N \ewght{t-1}{j} \adjfunc{t}{t}{\epart{t-1}{j}}}
\left( \ewghtfunc{t}(\epart{t-1}{\ell},x) G_{s,t,T}(x,h)  \right)^2 \rmd x \eqsp,
\\&= \frac{\sumwght{t-1}}{\sum_{j=1}^N \ewght{t-1}{j} \adjfunc{t}{t}{\epart{t-1}{j}}} \, \frac{1}{\sumwght{t-1}}\sum_{\ell=1}^N \ewght{t-1}{\ell} \Upsilon_{s,t,T}(\epart{t-1}{\ell},h) \eqsp,
\end{align*}
where
\begin{equation}
\label{eq:definition-A_N}
\Upsilon_{s,t,T}(x,h) \eqdef \adjfunc{t}{t}{x} \int \kiss{t}{t}(x,x') \ewghtfunc{t}^2(x,x') G^2_{s,t,T}(x',h) \rmd x' \eqsp.
\end{equation}
We will show, by applying Lemma \ref{lem:toutbete-1} that
\[
\frac{1}{\sumwght{t-1}}\sum_{\ell=1}^N \ewght{t-1}{\ell} \Upsilon_{s,t,T}(\epart{t-1}{\ell},h) \plim \post[\Xinit]{t-1}{t-1}\left( \upsilon_{s,t,T} (\cdot,h) \right)\eqsp,
\]
where $\upsilon_{s,t,T}$ is defined in \eqref{eq:definition-upsilon}.
\note{EM}{on suppose que $\Xset$ est un e.v.n}
To that purpose, we need to show that \textbf{(i)} there exists a constant $\upsilon_\infty$ such that $\esssup{\Upsilon_{s,t,T}(\cdot,h)} \leq \upsilon_\infty$ $\PP$-\as\, $\esssup{\upsilon_{s,t,T}(\cdot,h)} \leq \upsilon_\infty$, and \textbf{(ii)} for any constant $M \geq 0$, $\sup_{|x| \leq M} |\Upsilon_{s,t,T}(x,h) - \upsilon_{s,t,T}(x,h)| \plim 0$. Eq.~\eqref{eq:borne-ewghtfuncxG} shows that
$\esssup{\ewghtfunc{t}(\cdot) G_{s,t,T}(\cdot,h)} \leq 2 \esssup{\ewghtfunc{t}} \esssup{L_{t,t,T}(\cdot,\1)} \esssup{h}$, which implies that
\begin{equation}
\label{eq:borne-AN}
\esssup{\Upsilon_{s,t,T}(\cdot,h)} \leq 2  \esssup{\adjfunc{t}{t}{}} \esssup{\ewghtfunc{t}}^2 \esssup{L_{t,t,T}(\cdot,\1)}^2 \esssup{h}^2 \eqsp, \quad \PP-\as\ \eqsp.
\end{equation}
Similarly, the bound $\esssup{g_{s,t,T}} \leq \esssup{L_{t,t,T}(\cdot,\1)} \esssup{h}$ implies that
\begin{equation}
\label{eq:borne-a}
\esssup{\upsilon_{s,t,T}(\cdot,h)} \leq \esssup{\adjfunc{t}{t}{}} \esssup{\ewghtfunc{t}}^2 \esssup{L_{t,t,T}(\cdot,\1)}^2 \esssup{h}^2 \eqsp.
\end{equation}
The bounds \eqref{eq:borne-AN} and \eqref{eq:borne-a} imply \textbf{(i)}. To prove \textbf{(ii)}, first note that for all $M \geq 0$,
\begin{multline}
\label{eq:borne-Upsilon_N-a}
\sup_{|x| \leq M} |\Upsilon_{s,t,T}(x,h) - \upsilon_{s,t,T}(x,h)| \\
\leq \esssup{\adjfunc{t}{t}{}} \esssup{\ewghtfunc{t}} \int \sup_{|x| \leq M} \kiss{t}{t}(x,x') \left| G^2_{s,t,T}(x',h) - g^2_{s,t,T}(x',h) \right| \rmd x' \eqsp.
\end{multline}
Under (A\ref{assum:bound-proposal-kernel}), $\int \sup_{|x| \leq M} \kiss{t}{t}(x,x') \rmd x' < \infty$, and
since $\left| G^2_{s,t,T}(x',h) - g^2_{s,t,T}(x',h) \right| \leq 2  \esssup{L_{t,t,T}(\cdot,\1)}^2 \esssup{h}^2$, Lemma~\ref{lem:toutbete} shows that
\[
\int \sup_{|x| \leq M} \kiss{t}{t}(x,x') \left| G^2_{s,t,T}(x',h) - g^2_{s,t,T}(x',h) \right| \rmd x' \plim 0 \eqsp,
\]
provided that we can show that, for any given $x \in \Xset$, $G_{s,t,T}(x,h) \plim g_{s,t,T}(x,h)$. The definitions \eqref{eq:definition-Ft}
of the function $F_{s,t,T}$ and \eqref{eq:forward-filtering-backward-smoothing-alternate} of the smoothing distribution implies that
\[
F_{s,t,T}(x,h)= \frac{\post[\Xinit][hat]{s:t-1}{t-1}[q_{s,t-1}(\cdot,x) L_{s,t,T}([\cdot,x],h)]}{\post[\Xinit][hat]{t-1}{t-1}(q(\cdot,x))} \eqsp.
\]
Theorem~\ref{thm:Hoeffding-FFBS} show that $F_{s,t,T}(x,h) \plim g_{s,t,T}(x,h)$.
On the other hand, it follows from the definitions that $A_{s,t-1,T}(h) = \sumwght{t-1} \post[\Xinit][hat]{s:t}{t}(L_{s,t,T}(\cdot,h))$.\\
Since $\post[\Xinit]{s:t}{t}(L_{s,t,T}(\cdot,h))= 0$, this decomposition implies that $A_{s,t-1,T}(h) \plim 0$. Therefore,
$F_{s,t,T}(x,h) \plim g_{s,t,T}(x,h)$.

It remains to check the tightness condition \eqref{eq:triangular-3}. This property is straightforward since $|U_{M,i}|\leq N^{-1/2} \esssup{\ewghtfunc{t}} \esssup{h}$.

We now prove \eqref{eq:cltSi}. Using \eqref{eq:EspCond} and the definition of  $\post[\Xinit][hat]{s:T}{T}(h)$, we may write
\begin{multline*}
\sqrt{N}\left(\post[\Xinit][tilde]{s:T}{T}(h) -  \post[\Xinit]{s:T}{T}(h)\right)=  N^{-1/2} \sum_{\ell=1}^N \left[ h \left(\epart{s:T}{\chunk{J}{s}{T}^\ell}\right) - \CPE{h\left(\epartpredIndex{s:T}{J}\right)}{\mcf{T}}  \right] \\
+\sqrt{N}\left(\post[\Xinit][hat]{s:T}{T}(h) -  \post[\Xinit]{s:T}{T}(h)\right) \eqsp.
\end{multline*}
Note that since $(\chunk{J}{s}{T}^\ell)_{1 \leq \ell \leq N}$ are iid conditional to $\mcf{T}$, \eqref{eq:cltSi} follows from \eqref{eq:clt} and direct application of Corollary 11 in \cite{douc:moulines:2008} by noting that
\begin{multline*}
 N^{-1} \sum_{\ell=1}^N \CPE{\left\{h \left(\epart{s:T}{\chunk{J}{s}{T}^\ell}\right) - \CPE{h\left(\epartpredIndex{s:T}{J}\right)}{\mcf{T}} \right\}^2}{\mcf{T}} \\ =\post[\Xinit][hat]{s:T}{T}^2\left[h-  \post[\Xinit][hat]{s:T}{T}(h)\right]
 \plim \post[\Xinit]{s:T}{T}^2 \left[h-  \post[\Xinit]{s:T}{T}(h)\right]
\end{multline*}
\end{proof}

\section{Time-uniform deviation inequality}
\label{sec:TimeUniformExponentialFFBS}
In this section, we study the long-term behavior of the marginal fixed-interval smoothing distribution estimator.
For that purpose, it is required to impose a form of mixing condition on the Markov transition kernel. For simplicity, we consider conditions which are similar to the ones used in \cite[chapter 7.4]{delmoral:2004} or \cite[chapter 4]{cappe:moulines:ryden:2005}; these conditions can be relaxed, but at the expense of many technical problems. This condition requires that the transition kernel is strongly mixing in the sense that
\begin{assum}
\label{assum:strong-mixing-condition}
There exist two constants $0 < \sigma_- \leq \sigma_+ < \infty$, such that, for any $(x,x') \in \Xset \times \Xset$,
\begin{equation}
\label{eq:minorq}
\sigma_- \leq q(x,x') \leq \sigma_+\eqsp.
\end{equation}
In addition, there exists a constant $c_- > 0$ such that, $\int \Xinit(\rmd x_0) g_0(x_0) \geq c_-$ and
for all $t \geq 1$,
\begin{equation}
\label{eq:minorg}
\inf_{x \in \Xset} \int q(x,x') g_t(x') \rmd x' \geq c_- > 0 \eqsp.
\end{equation}
\end{assum}
Note that assumption \ref{assum:strong-mixing-condition} implies that $\nu(\Xset)<\infty$; in the sequel, we will consider without loss of generality that $\nu(\Xset)=1$. Moreover, the average number of simulations required in the accept-reject mechanism per sample of the FFBSi algorithm is bounded by $\sigma^+/\sigma^-$. 
An important consequence of the uniform ergodicity condition is the forgetting of either the initial or the final conditions. The
 following Proposition extends some of the results obtained initially in  \cite{delmoral:guionnet:1998} and later extended in \cite{delmoral:ledoux:miclo:2003} (see also \cite[Chapter 7]{delmoral:2004} and \cite[Chapter 2]{cappe:moulines:ryden:2005}).
Define
\begin{multline}\label{eq:definition-lst}
\ell_{s,t}(x_s,x_t) \eqdef \idotsint q(x_s, x_{s+1}) g_{s+1}(x_{s+1})\\
\times \prod_{u=s+1}^{t-1} q(x_u, x_{u+1}) g_{u+1}(x_{u+1}) \rmd \chunk{x}{s+1}{t-1}
\end{multline}
for $s<t$, and $\ell_{t,t}(x_s, x_t) \eqdef \delta_{x_s}(x_t)$, so that
\[
L_{s,t}(x_s, h) = \int \ell_{s,t}(x_s, x_t) h(x_t)\rmd x_t \eqsp.
\]
\begin{prop}
\label{prop:forgetting-initial-final-conditions}
Assume (A\ref{assum:strong-mixing-condition}). Then,  for all distributions $\Xinit$, $\Xinit'$ and for all $s\leq t$ and any bounded measurable functions $h$,
\begin{multline}
\label{eq:oubli-filtre-forward}
\left| \frac{\iint \Xinit(\rmd x_s) \ell_{s,t}(x_s,x_t) h(x_t) \rmd x_t}{\iint \Xinit(\rmd x_s) \ell_{s,t}(x_s,x_t) \rmd x_t} -
\frac{\iint \Xinit'(\rmd x_s) \ell_{s,t}(x_s,x_t) h(x_t) \rmd x_t}{\iint \Xinit'(\rmd x_s) \ell_{s,t}(x_s,x_t) \rmd x_t} \right|\\
\leq \rho^{t-s}\oscnorm{h} \eqsp,
\end{multline}
where $\rho \eqdef 1 -\sigma_-/\sigma_+$. In addition, for any bounded non-negative measurable functions $f$ and $f'$,
\begin{multline}
\label{eq:oubli-filtre-backward}
\left| \frac{\iint \Xinit(\rmd x_s) h(x_s) \ell_{s,t}(x_s,x_t) f(x_t) \rmd x_t }{\iint \Xinit(\rmd x_s) \ell_{s,t}(x_s,x_t) f(x_t) \rmd x_t} -
\frac{\iint \Xinit(\rmd x_s) h(x_s) \ell_{s,t}(x_s,x_t) f'(x_t) \rmd x_t}{\iint \Xinit(\rmd x_s) \ell_{s,t}(x_s,x_t) f'(x_t) \rmd x_t} \right|\\
\leq \rho^{t-s}\oscnorm{h} \eqsp,
\end{multline}
as soon as the denominators are non-zero.
\end{prop}
\begin{proof}
\newcommand{\backkernel}[2]{B_{#1,#2}}
\newcommand{\dbackkernel}[2]{b_{#1,#2}}
The first statement is well-known; see for example \cite{delmoral:2004} and \cite[Proposition 4.3.23]{cappe:moulines:ryden:2005}. To check the second statement, denote by $\backkernel{\Xinit}{s,t}$ the backward smoothing kernel, defined for $s \leq t$ and for any real-valued, measurable function $\psi$ on $\Xset^2$ by
$$\PE_{X_s\sim\Xinit}\left[\psi(X_{t-1}, X_t) | Y_{s:t}\right] = \int \psi(x_{t-1}, x_t)\post[\Xinit]{s:t}{s:t}(x_{t}) \backkernel{\Xinit}{s,t}(x_t, \rmd x_{t-1})\eqsp,$$
where $\post[\Xinit]{s:t-1}{s:t-1}$ is defined in \eqref{eq:convNot}.
This kernel is absolutely continuous \wrt\ the dominating measure $\lambda$ and its density is given by
$$\dbackkernel{\Xinit}{s,t}(x_t, x_{t-1}) = \frac{\post[\Xinit]{s:t-1}{s:t-1}(x_{t-1}) q (x_{t-1}, x_t)} {\int \post[\Xinit]{s:t-1}{s:t-1}(x) q (x, x_t) \rmd x} \eqsp.$$
 Under (A\ref{assum:strong-mixing-condition}), this transition density is lower bounded by
$$ \dbackkernel{\Xinit}{s,t}(x_t, x_{t-1}) \geq \frac{\sigma_-}
{\int \post[\Xinit]{s:t-1}{s:t-1}(x) \sigma_+  \rmd x} = \frac{\sigma_-}{\sigma_+} \eqsp.$$
Since $\oscnorm{\backkernel{\Xinit}{s,t}(\cdot, h)} \leq \rho \oscnorm{h}$, it follows that
\begin{equation}
\label{eq:bound-oscnorm}
\oscnorm{\backkernel{\Xinit}{s,t}\ldots\backkernel{\Xinit}{s,s+1}(\cdot, h)} \leq \rho^{t-s} \oscnorm{h} \eqsp.
\end{equation}
Note that
\begin{align*}
&\left|\frac{\iint \Xinit(\rmd x_s) h(x_s) \ell_{s,t}(x_s,x_t) f(x_t) \rmd x_t }{\iint \Xinit(\rmd x_s)  \ell_{s,t}(x_s,x_t) f(x_t) \rmd x_t }  - \frac{\iint \Xinit(\rmd x_s) h(x_s) \ell_{s,t}(x_s,x_t) f'(x_t) \rmd x_t }{\iint \Xinit(\rmd x_s)  \ell_{s,t}(x_s,x_t) f'(x_t) \rmd x_t }\right|\\
 &=  \left|\frac{\post[\Xinit]{t}{s:t}(\backkernel{\Xinit}{s,t}\dots\backkernel{\Xinit}{s,s+1}(\cdot, h)f(\cdot))}{\post[\Xinit]{t}{s:t}(f(\cdot))} - \frac{\post[\Xinit]{t}{s:t}(\backkernel{\Xinit}{s,t}\dots\backkernel{\Xinit}{s,s+1}(\cdot, h)f'(\cdot))}{\post[\Xinit]{t}{s:t}(f'(\cdot))}\right|
 \\
 &=  \left| \mu_f\left[\backkernel{\Xinit}{s,t}\dots\backkernel{\Xinit}{s,s+1}(\cdot, h)\right] - \mu_{f'}\left[\backkernel{\Xinit}{s,t}\dots\backkernel{\Xinit}{s,s+1}(\cdot, h)\right] \right|
\end{align*}
with, for any $A \in \Xsigma$,  $\mu_f(A)  \eqdef \post[\Xinit]{t}{s:t}(\1_A f)/ \post[\Xinit]{t}{s:t}(f)$.
Therefore, since for any probabilities $\mu$ and $\mu'$ on $\Xsigma$ and any measurable function $\psi$,
$|\mu(\psi) - \mu'(\psi)| \leq \oscnorm{\psi}$, \eqref{eq:bound-oscnorm} shows that
$$
\left| \mu_f\left[\backkernel{\Xinit}{s,t}\dots\backkernel{\Xinit}{s,s+1}(\cdot, h)\right] - \mu_{f'}\left[\backkernel{\Xinit}{s,t}\dots\backkernel{\Xinit}{s,s+1}(\cdot, h)\right] \right|\leq \rho^{t-s} \oscnorm{h} \eqsp.
$$
\end{proof}

The goal of this section consists in establishing, under the assumptions mentioned above, that the FFBS approximation of the \emph{marginal} fixed interval smoothing probability satisfies an exponential  deviation inequality with constants that are uniform in time.

 The first step in the proof consists in showing a time-uniform deviation inequality for the auxiliary particle filter. Here again, the proof of this result could be adapted from \cite[Section 7.4.3]{delmoral:2004}. For the sake of clarity, we present a self-contained proof, which is valid under assumptions that are weaker than those used in \cite[Chapter 7]{delmoral:2004}.
\begin{prop}
\label{prop:time-uniform-exponential-inequality-forward}
Assume that A\ref{assum:bound-likelihood}--\ref{assum:strong-mixing-condition} hold with $T= \infty$. Then, the filtering distribution satisfies a time-uniform exponential deviation inequality, \ie\
there exist constants $B$ and $C$ such that, for all integers $N$ and $t \geq 0$, all measurable functions $h$ and all $\epsilon > 0$,
\begin{equation}
\label{eq:TU-Hoeffding-Filtering-unnormalized}
\PP\left[ \left| N^{-1}\sum_{i=1}^N \swght{t}{i} h(\epartpred{t}{i}) - \frac{\post[\Xinit]{t}{t-1}(g_t h)}{\post[\Xinit]{t-1}{t-1}(\adjfunc{t}{t}{})}\right| \geq \epsilon \right]
\leq B \rme^{-C N \epsilon^2 / \esssup{h}^2} \eqsp,
\end{equation}
\begin{equation}
\PP\left[ \left| \post[\Xinit][hat]{t}{t}(h) - \post[\Xinit]{t}{t}(h) \right| \geq \epsilon \right] \leq B \rme^{-C N \epsilon^2 / \oscnorm[2]{h}} \eqsp. \label{eq:TU-Hoeffding-Filtering-normalized}
\end{equation}
\end{prop}

\begin{proof}
We first prove \eqref{eq:TU-Hoeffding-Filtering-normalized}.
Without loss of generality, we will assume that $\post[\Xinit]{t}{t}(h)=0$.
Similar to \cite[Eq. (7.24)]{delmoral:2004}, the quantity $\post[\Xinit][hat]{t}{t}(h) $ is decomposed as,
\begin{equation}
\label{eq:decomposition-time-uniform}
\post[\Xinit][hat]{t}{t}(h)  = \sum_{s=1}^t \left( \frac{B_{s,t}(h)}{B_{s,t}(\1)} - \frac{B_{s-1,t}(h)}{B_{s-1,t}(\1)}\right) + \frac{B_{0,t}(h)}{B_{0,t}(\1)} \eqsp,
\end{equation}
where
\begin{equation}
\label{eq:definition-Bst}
B_{s,t}(h)= N^{-1} \sum_{i=1}^N \ewght{s}{i} \frac{L_{s,t}(\epart{s}{i},h)}{\esssup{L_{s,t}(\cdot,\1)}} \eqsp.
\end{equation}
We first establish some exponential inequality for $B_{0,t}(h)/B_{0,t}(\1)$ where the dependence in $t$ will be explicitly expressed. For that purpose, we will apply Lemma \ref{lem:inegEssentielle} by successively checking Conditions \eqref{item:inegEssentielle-borne}, \eqref{item:inegEssentielle-expo1}, and  \eqref{item:inegEssentielle-expo2},  with
$$
\begin{cases}
a_N \eqdef B_{0,t}(h)\\
b_N\eqdef B_{0,t}(\1)\\
c_N\eqdef 0\\
b \eqdef \int \Xinit(\rmd x_0) g_0(x_0) \frac{L_{0,t}(x_0,\1)}{\esssup{L_{0,t}(\cdot,\1)}}\\
\beta \eqdef \frac{\sigma_-}{\sigma_+}  \int \Xinit(\rmd x_0) g_0(x_0) \eqsp.
\end{cases}
$$
Under the strong mixing condition A\ref{assum:strong-mixing-condition}, it may be shown that (see \cite[chapter 4]{delmoral:2004}
or \cite[section 4.3.3]{cappe:moulines:ryden:2005})
\begin{equation}
\label{eq:lower-upper-bound}
\frac{\sigma_-}{\sigma_+} \leq \frac{L_{s,t}(\epart{}{},\1)}{\esssup{L_{s,t}(\cdot,\1)}} \leq \frac{\sigma_+}{\sigma_-} \eqsp,
\end{equation}
which implies that $b \geq \beta$. Since $\post[\Xinit]{t}{t}(h)=0$, Eqs.~\eqref{eq:interpretation-L} and \eqref{eq:oubli-filtre-forward} imply
\begin{multline}
\label{eq:decomp-1}
\left|\frac{a_N}{b_N}\right|=\left|\frac{B_{0,t}(h)}{B_{0,t}(\1)} - \post[\Xinit]{t}{t}(h)\right| \\
= \left|\frac{\sum_{i=1}^N \ewght{0}{i} L_{0,t}(\epart{0}{i},h) }{\sum_{i=1}^N \ewght{0}{i} L_{0,t}(\epart{0}{i},\1)} -
\frac{\int \Xinit(\rmd x_0) g_0(x_0) L_{0,t}(x_0,h) }{\int \Xinit(\rmd x_0) g_0(x_0) L_{0,t}(x_0,\1)} \right| \leq \rho^t \oscnorm{h}\eqsp.
\end{multline}
This shows condition (I) with $M=\rho^t \esssup{h}$. We now turn to condition (II). We have
$$
b_N-b= N^{-1} \sum_{i=1}^N \ewght{0}{i} \frac{L_{0,t}(\epart{0}{i},\1)}{\esssup{L_{0,t}(\cdot,\1)}}- \int  \XinitIS{0}(\rmd x_0)  \ewghtfunc{0}(x_0) g_0(x_0) \frac{L_{0,t}(x_0,\1)}{\esssup{L_{0,t}(\cdot,\1)}}  \eqsp.
$$
Since $\left|\ewght{0}{i} \frac{L_{0,t}(\epart{0}{i},\1)}{\esssup{L_{0,t}(\cdot,\1)}}\right|\leq \frac{\sigma_+}{\sigma_-} \esssup{\ewghtfunc{0}}$, we have by Hoeffding's inequality
$$
\PP\left[ \left| b_N-b \right| \geq \epsilon\right] \leq B \exp\left(-C N\epsilon^2\right) \eqsp,
$$
where the constants $B$ and $C$ do not depend on $t$. This shows condition (II). We now check condition (III). We have
$$
a_N-(c_N/d_N)b_N=a_N=N^{-1} \sum_{i=1}^N \ewght{0}{i} \frac{L_{0,t}(\epart{0}{i},h)}{\esssup{L_{0,t}(\cdot,\1)}} \eqsp.
$$
Now, as $\post[\Xinit]{t}{t}(h)=0$ implies $\int \Xinit(\rmd x) L_{0,t}(x,h)=0$, it holds that $\PE(a_N)=0$. Moreover,
\begin{multline*}
\left| \ewght{0}{i} \frac{L_{0,t}(\epart{0}{i},h)}{\esssup{L_{0,t}(\cdot,\1)}}\right|\leq  \esssup{\ewghtfunc{0}}\left| \frac{L_{0,t}(\epart{0}{i},\1)}{{\esssup{L_{0,t}(\cdot,\1)}}}\left( \frac{L_{0,t}(\epart{0}{i},h)}{L_{0,t}(\epart{0}{i},\1)}- \post[\Xinit]{t}{t}(h)\right)\right|\\ \leq \esssup{\ewghtfunc{0}}\frac{\sigma_+}{\sigma_-} \rho^t \oscnorm{h} \eqsp,
\end{multline*}
using  \eqref{eq:interpretation-L} and \eqref{eq:oubli-filtre-forward}. Condition (III) follows from Hoeffding's inequality. Then,  Lemma \ref{lem:inegEssentielle} gives
$$
\PP\left[\left|B_{0,t}(h)/B_{0,t}(\1)\right|>\epsilon\right] \leq B \rme^{-C N \epsilon^2/(\rho^t \oscnorm{h})^2} \eqsp,
$$
where the constants $B$ and $C$ do not depend on $t$.

We now consider for $1 \leq s \leq t$ the difference $B_{s,t}(h)/B_{s,t}(\1) - B_{s-1,t}(h)/B_{s-1,t}(\1)$, where $B_{s,t}$ is defined in \eqref{eq:definition-Bst}. We again use Lemma \ref{lem:inegEssentielle} where $\PP(\cdot)=\CPP{\cdot}{\mcf[]{s-1}}$,  $a_N=B_{s,t}(h)$, $b_N=B_{s,t}(\1)$, $c_N=B_{s-1,t}(h)$, $d_N=B_{s-1,t}(\1)$,
$$
b=\frac{  \sum_{i=1}^N \ewght{s-1}{i} \int Q(\epart{s-1}{i},\rmd x) g_s(x)  L_{0,s}(x,\1)}{\esssup{L_{0,s}(\cdot,\1) } \sum_{\ell=1}^N \ewght{s-1}{\ell} \adjfunc{s}{s}{\epart{s-1}{\ell}}} \eqsp, \quad \mbox{and} \quad  \beta=\frac{c_- \sigma_-}{\sigma_+ \esssup{\adjfunc{s}{s}{}}}  \eqsp,
$$
where $\sigma_-$ and $c_-$ are defined in \eqref{eq:minorq} and \eqref{eq:minorg}, respectively.
It appears using \eqref{eq:lower-upper-bound} and (A\ref{assum:strong-mixing-condition}) that $b\geq \beta$.
Moreover,
\begin{equation}
\label{eq:borne-31-12-1}
\left| \frac{a_N}{b_N} - \frac{c_N}{d_N} \right| =\left| \frac{\sum_{i=1}^N \ewght{s}{i} L_{s,t}(\epart{s}{i},h) }{\sum_{i=1}^N \ewght{s}{i} L_{s,t}(\epart{s}{i},\1)}-\frac{\sum_{i=1}^N \ewght{s-1}{i} L_{s-1,t}(\epart{s-1}{i},h) }{\sum_{i=1}^N \ewght{s-1}{i} L_{s-1,t}(\epart{s-1}{i},\1)}\right| \leq  \rho^{t-s}\oscnorm{h} \eqsp,
\end{equation}
showing condition \eqref{item:inegEssentielle-borne} with $M= \rho^{t-s}\oscnorm{h}$. We now check condition \eqref{item:inegEssentielle-expo1}. By \eqref{eq:relat1}, we have
$$
b_N-b= N^{-1} \sum_{i=1}^N \ewght{s}{i} \frac{L_{s,t}(\epart{s}{i},\1)}{\esssup{L_{s,t}(\cdot,\1)}}- \CPE{\ewght{s}{1}\frac{L_{s,t}(\epart{s}{1},\1)}{\esssup{L_{s,t}(\cdot,\1)}}}{\mcf{s-1}}{} \eqsp.
$$
Thus, since $\left| \ewght{s}{i} {L_{s,t}(\epart{s}{i},\1)}/{\esssup{L_{s,t}(\cdot,\1)}}\right| \leq \sup_t \esssup{\ewghtfunc{t}} \sigma_+/\sigma_-$, we have by conditional Hoeffding's inequality
$$
\CPP{|b_N-b|>\epsilon}{\mcf{s-1}{}} \leq B \rme^{-C N \epsilon^2/(\oscnorm{h})^2}
$$
showing condition \eqref{item:inegEssentielle-expo1} with constants which do not depend on $s$. Moreover, write $
a_N-\frac{c_N}{d_N}b_N =  N^{-1}\sum_{\ell=1}^N  \eta^\ell
$ where
$$
\eta^\ell\eqdef \ewght{s}{\ell} \frac{L_{s,t}(\epart{s}{\ell},h)}{\esssup{L_{s,t}(\cdot,\1)}} - \frac{\sum_{i=1}^N \ewght{s-1}{i} L_{s-1,t}(\epart{s-1}{i},h) }{\sum_{i=1}^N \ewght{s-1}{i} L_{s-1,t}(\epart{s-1}{i},\1)}\left(\ewght{s}{\ell} \frac{L_{s,t}(\epart{s}{\ell},\1)}{\esssup{L_{s,t}(\cdot,\1)}}\right)\eqsp.
$$
Since $\{(\epart{t}{\ell},\ewght{t}{\ell})\}_{\ell=1}^N$ are \iid\ conditionally to the $\sigma$-field  $\mcf{t-1}$, we have that $\{\eta^\ell\}_{\ell=1}^{N}$ are also \iid\ conditionally to $\mcf{t-1}$. Moreover, it can be easily checked using \eqref{eq:relat1} that $\CPE{\eta^1}{\mcf{t-1}}=0$. In order to apply the conditional Hoeffding inequality, we need to check that $\eta^\ell$ is bounded. In fact, using \eqref{eq:interpretation-L} and \eqref{eq:oubli-filtre-forward},
$$
|\eta^\ell|  =\frac{L_{s,t}(\epart{s}{i},\1)}{\esssup{L_{s,t}(\cdot,\1)}} \left| \frac{L_{s,t}(\epart{s}{\ell},h)}{L_{s,t}(\epart{s}{\ell},\1)} - \frac{\sum_{i=1}^N \ewght{s-1}{i} L_{s-1,t}(\epart{s-1}{i},h) }{\sum_{i=1}^N \ewght{s-1}{i} L_{s-1,t}(\epart{s-1}{i},\1)}\right| \leq \frac{\sigma_+}{\sigma_-} \rho^{t-s}\oscnorm{h}
$$
Consequently,
\begin{multline*}
\CPP{\left|a_N-\frac{c_N}{d_N}b_N\right| >\epsilon}{\mcf{s-1}{}}= \CPP{\left| N^{-1}\sum_{\ell=1}^N  \eta^\ell\right|>\epsilon}{\mcf{s-1}{}} \\
\leq B \exp\left\{-CN \left(\frac{\epsilon}{\rho^{t-s}\oscnorm{h}}\right)^2\right\} \eqsp,
\end{multline*}
where the constants $B$ and $C$ do not depend on $s$. This shows condition \eqref{item:inegEssentielle-expo2}. Finally by Lemma \ref{lem:inegEssentielle},
$$
\CPP{\left|\frac{B_{s,t}(h)}{B_{s,t}(\1)} - \frac{B_{s-1,t}(h)}{B_{s-1,t}(\1)}\right| >\epsilon}{\mcf{s-1}{}} \leq B \exp\left\{-CN \left(\frac{\epsilon}{\rho^{t-s}\oscnorm{h}}\right)^2\right\}
$$
The proof is concluded by using Lemma \ref{lem:pasfor}.
\end{proof}

We now show that the time uniform deviation inequality for the filtering estimator extends, under the mixing assumption (A\ref{assum:strong-mixing-condition}) on the Markov kernel $Q$, to the FFBS smoothing estimator.
The key result to establish a time uniform bound for the FFBS smoothing estimator  is the following Proposition, which establishes the uniform ergodicity of the particle approximation of the backward kernel.
\begin{lem}
\label{lem:uniform-forgetting-backward-kernel}
Assume A\ref{assum:strong-mixing-condition}.
Then, for any probability distributions $\mu$ and $\mu'$ on the set $\{1,\dots,N\}$, any integers $0 \leq s < t $ and any function $h$
on $\{1,\dots,N\}$,
\[
\left| \sum_{i_{s:t}=1}^N h(i_s) \swghtcIndex{s:t}{i} \left\{ \mu(i_t) - \mu'(i_t) \right\} \right| \leq \oscnorm{h} \rho^{t-s} \eqsp,
\]
where $\swghtcIndex{s:t}{i}$ is defined in \eqref{eq:definition-swghtcIndex}.
\end{lem}
\begin{proof}
For $u \in \{s+1, \dots, t\}$, define $W_u$ the $N \times N$ matrix with entries
\[
W_u^{i,j}= \frac{\ewght{u-1}{j} q(\epart{u-1}{j},\epart{u}{i})}{\sum_{\ell=1}^N \ewght{u-1}{\ell} q(\epart{u-1}{\ell},\epart{u}{i})} \eqsp, \quad 1 \leq i, j \leq N \eqsp.
\]
The matrix $W_u$ can be interpreted as the Markov transition matrix of a non-homogeneous Markov chain on the state-space $\{1,\dots,N\}$ (which may be seen as the particle approximation of the backward kernel \eqref{eq:backward-kernel}). Using this notation, for any probability distribution $\mu$ on $\{1,\dots,N\}$ and any function $h$ on $\{1,\dots,N\}$, the sum $\sum_{i_{s:t}=1}^N h(i_s) \swghtcIndex{s:t}{i} \mu(i_t)$ may be interpreted as the expectation of the function $h$ under the marginal distribution at time $t-s$ of a non-homogeneous Markov chain started at time $0$ from the initial distribution $\mu$ and driven the transition matrix $W_{t}$, $W_{t-1}$, $\dots$:
\[
\sum_{i_{s:t}=1}^N h(i_s) \swghtcIndex{s:t}{i} \mu(i_t)= \sum_{i_{s:t}=1}^N \mu(i_t) W_{t}^{i_t,i_{t-1}} \dots W_{s+1}^{i_{s+1},i_s} h(i_s) \eqsp.
\]
Under (A\ref{assum:strong-mixing-condition}), the entries of these transition kernels are lower-bounded by $\sigma_-/\sigma_+$. Therefore, the Dobrushin coefficient of each transition matrix $W_u$, $u \in \{s+1,\dots, t\}$ is upper bounded by $\rho$ (see \cite{dobrushin:1956}). The result follows.
\end{proof}
 We then show that the existence of time-uniform exponential deviation inequality for the auxiliary particle filter approximation of the filtering distribution extends to the FFBS smoothing estimator.

\begin{thm}
\label{theo:Hoeffding-uniform}
Assume A\ref{assum:bound-likelihood}--\ref{assum:strong-mixing-condition} hold with $T= \infty$.
Then, there exist constants $0\leq B,\ C< \infty$ such that for all integers $N$, $s$, and $T$, $s \leq T$, all $\epsilon > 0$,
\begin{align}
&\PP\left[ \left| \post[\Xinit][hat]{s}{T}(h) - \post[\Xinit]{s}{T}(h) \right| \geq \epsilon \right] \leq B\rme^{-C N \epsilon^2 / \oscnorm[2]{h}} \eqsp, \label{eq:TU-Hoeffding-smoothing-normalized}\\
& \PP\left[ \left| \post[\Xinit][tilde]{s}{T}(h) - \post[\Xinit]{s}{T}(h) \right| \geq \epsilon \right] \leq B\rme^{-C N \epsilon^2 / \oscnorm[2]{h}} \eqsp, \label{eq:TU-Hoeffding-smoothing-normalized-2}
\end{align}
where $\post[\Xinit][hat]{s}{T}(h)$ and $\post[\Xinit][tilde]{s}{T}(h)$ are defined in \eqref{eq:forward-filtering-backward-smoothing} and \eqref{eq:FFBSi:estimator}.
\end{thm}

\begin{proof}
\eqref{eq:TU-Hoeffding-smoothing-normalized-2} follows from \eqref{eq:TU-Hoeffding-smoothing-normalized} along the same lines as in Theorem \ref{thm:Hoeffding-FFBS}.
We use the notations of Theorem \ref{thm:Hoeffding-FFBS}. Let $h$ be a function defined on $\Xset$ and $s,T$ be positive integers such that $s\leq T$. Without loss of generality, we assume that $\post[\Xinit]{s}{T}(h)= 0$. We will denote by
\begin{equation}
\label{eq:definition-barh}
\bar{h}: (x_s,\ldots,x_T)\mapsto h(x_s) \eqsp.
\end{equation}
For $s \in \{0,\dots, T\}$, consider again the following decomposition
\[
\post[\Xinit][hat]{s}{T}(h)= \frac{A_{s,T,T}(\bar{h})}{A_{s,T,T}(\1)} = \frac{A_{s,s,T}(\bar{h})}{A_{s,s,T}(\1)} + \sum_{t=s+1}^T \left\{ \frac{A_{s,t,T}(\bar{h})}{A_{s,t,T}(\1)} - \frac{A_{s,t-1,T}(\bar{h})}{A_{s,t-1,T}(\1)}\right\} \eqsp,
\]
where $\bar{h}$, $A_{s,s,T}$, and $A_{s,t,T}$ are defined in \eqref{eq:definition-barh}, \eqref{eq:definition-Fs-As} and \eqref{eq:definition-Ast}, respectively. Note that  $\bar{h}$ depends of $\epartpredIndex{s:t}{i}$ only through its first component $\epart{s}{i_s}$; therefore, it follows from the definition \eqref{eq:definition-Lt} of $L_{s,t,T}$  that $L_{s,t,T}(\epartpredIndex{s:t}{i},\bar{h})= h(\epart{s}{i_s}) L_{t,t,T}(\epart{t}{i_t},\1)$. In particular, $L_{s,t,T}(\epartpredIndex{s:t}{i},\1)= L_{t,t,T}(\epart{t}{i_t},\1)$.\\
We first consider the term $A_{s,s,T}(\bar{h})/A_{s,s,T}(\1)$. It follows from the definition that
\[
A_{s,s,T}(\bar{h})=  \sum_{\ell=1}^N \ewght{s}{\ell} F_{s,s,T}(\epart{s}{\ell},\bar{h}) =  \sum_{\ell=1}^N \ewght{s}{\ell} h(\epart{s}{\ell}) L_{s,s,T}(\epart{s}{\ell},\1) \eqsp.
\]
We apply Lemma \ref{lem:inegEssentielle} with
$$
\begin{cases}
a_N= \sumwght{s}^{-1} A_{s,s,T}(\bar{h})/\esssup{L_{s,s,T}(\cdot,\1)}\\
b_N= \sumwght{s}^{-1} A_{s,s,T}(\1)/\esssup{L_{s,s,T}(\cdot,\1)}\\
c_N=0\\
b=\post[\Xinit]{s}{s}[L_{s,s,T}(\cdot,\1)]/\esssup{L_{s,s,T}(\cdot,\1)}\\
\beta=\sigma_-/\sigma_+\eqsp.
\end{cases}
$$
Using the definition \eqref{eq:definition-Lt} and (A\ref{assum:strong-mixing-condition}), for any $s \in \{0,\dots, T\}$,
\begin{align}
&\label{eq:majorationL}
\frac{\sigma_-}{\sigma_+} \leq \frac{L_{s,s,T}(\epart{}{},\1)}{\esssup{L_{s,s,T}(\cdot,\1)}} \leq \frac{\sigma_+}{\sigma_-} \eqsp.%\\ \nonumber
%&= \frac{\idotsint q(\epart{}{},x_{t+1}) g_{t+1}(x_{t+1}) \prod_{u=t+1}^{T-1} q(x_u,x_{u+1}) g_{u+1}(x_{u+1}) \rmd x_{t+1:T}}
%{\idotsint \post[\Xinit]{t}{t}(x_t) q(x_t,x_{t+1}) g_{t+1}(x_{t+1}) \prod_{u=t+1}^{T-1} q(x_u,x_{u+1}) g_{u+1}(x_{u+1}) \rmd x_{t+1:T} }
%\\ \nonumber
%&\leq \frac{\sigma_+}{\sigma_-} \frac{ \idotsint (x_{t+1}) g_{t+1}(x_{t+1}) \prod_{u=t+1}^{T-1} q(x_u,x_{u+1}) g_{u+1}(x_{u+1}) \rmd x_{t+1:T}}{\idotsint \minorprob(x_{t+1}) g_{t+1}(x_{t+1}) \prod_{u=t+1}^{T-1} q(x_u,x_{u+1}) g_{u+1}(x_{u+1}) \rmd x_{t+1:T}} = \frac{\sigma_+}{\sigma_-} \eqsp.
\end{align}
Therefore $b \geq \beta$. Then, note that $|a_N/b_N| \leq \esssup{h}$; therefore condition \eqref{item:inegEssentielle-borne} is satisfied with $M= \esssup{h}$.
 We now check condition \eqref{item:inegEssentielle-expo1}. We have
 $$
 b_N -b =\post[\Xinit][hat]{t}{t}\left[\frac{L_{s,s,T}(\cdot,\1)}{\esssup{L_{s,s,T}(\cdot,\1)}}\right] - \post[\Xinit]{t}{t}\left[\frac{L_{s,s,T}(\cdot,\1)}{\esssup{L_{s,s,T}(\cdot,\1)}}\right]\eqsp.
 $$
 Inequalities \eqref{eq:TU-Hoeffding-Filtering-normalized} and \eqref{eq:majorationL} show that there exists constants $B$ and $C$ such that for any $\epsilon>0$ and all positive integers $s\leq T$,
\[
\PP\left( \left| b_N -b \right| \geq \epsilon \right) \leq B \rme^{-C N \epsilon^2} \eqsp.
\]
Hence, condition \eqref{item:inegEssentielle-expo1} is satisfied. Moreover,
$$
a_N-\frac{c_N}{d_N}b_N=a_N= \sumwght{s}^{-1} \sum_{\ell=1}^N \ewght{s}{\ell} G_s(\epart{s}{\ell})\eqsp, \quad \mbox{where} \quad G_{s}(\epart{}{})= h(\epart{}{}) \frac{L_{s,s,T}(\epart{}{},\1)}{\esssup{L_{s,s,T}(\cdot,\1)}}\eqsp.
$$
 Using the definition \eqref{eq:definition-Lt} of $L_{s,s,T}$,
\[
\post[\Xinit]{s}{T}(h)= \frac{\post[\Xinit]{s}{s}[h(\cdot) L_{s,s,T}(\cdot,\1)]}{\post[\Xinit]{s}{s}[L_{s,s,T}(\cdot,\1)]} \eqsp.
\]
The condition $\post[\Xinit]{s}{T}(h)=0$ therefore implies that $\post[\Xinit]{s}{s}(G_s)=0$. On the other hand, using \eqref{eq:majorationL}, $\esssup{G_s} \leq \esssup{h} \sigma_+/\sigma_-$.
Hence, by \eqref{eq:TU-Hoeffding-Filtering-normalized},
\[
\PP\left[ \left| a_N-\frac{c_N}{d_N}b_N \right| \geq \epsilon \right] \leq B \rme^{-C N \epsilon^2 / \oscnorm[2]{h}
} \eqsp,
\]
for some $B$ and $C$ which do not depend on $s$ nor $T$.
Hence condition \eqref{item:inegEssentielle-expo2} is satisfied. Combining the result above, Lemma \ref{lem:inegEssentielle} therefore shows that,
\[
\PP\left[ \left| \frac{A_{s,s,T}(\bar{h})}{A_{s,s,T}(\1)} \right| \geq \epsilon \right] \leq 2B \exp\left(-C N \epsilon^2/ \oscnorm[2]{h}\right) \eqsp.
\]
We now consider the term $A_{s,t,T}(\bar{h})/A_{s,t,T}(\1) - A_{s,t-1,T}(\bar{h})/A_{s,t-1,T}(\1)$ for $t>s$. For that purpose, we use Lemma \ref{lem:inegEssentielle} with
$$
\begin{cases}
a_N= N^{-1} A_{s,t,T}(\bar{h})/\esssup{L_{t,t,T}(\cdot,\1)}\\
b_N= N^{-1} A_{s,t,T}(\1)/\esssup{L_{t,t,T}(\cdot,\1)} \\
c_N= A_{s,t-1,T}(\bar{h})\\
d_N = A_{s,t-1,T}(\1)\\
b= \post[\Xinit]{t-1}{t-1}[L_{t-1,t-1,T}(\cdot,\1)]/(\esssup{L_{t,t,T}(\cdot,\1)}\post[\Xinit]{t-1}{t-1}(\adjfunc{t}{t}{}))\\
\beta= c_- \sigma_- / \left(\sigma_+ \esssup{\adjfunc{s}{s}{}}\right)
\end{cases}
$$
Inequality \eqref{eq:majorationL} and \eqref{eq:minorg} directly imply that $b\geq \beta$. Moreover,
\begin{align}
\nonumber
&\frac{a_N}{b_N} - \frac{c_N}{d_N}
= \sum_{i_{s:t}=1}^N h(\epart{s}{i_s}) \swghtcIndex{s:t-1}{i} \frac{\ewght{t-1}{i_{t-1}} q(\epart{t-1}{i_{t-1}},\epart{t}{i_t})}{\sum_{\ell=1}^N \ewght{t-1}{\ell} q(\epart{t-1}{\ell},\epart{t}{i_t})}
\frac{\ewght{t}{i_t} L_{t,t,T}(\epart{t}{i_t},\1)}{\sum_{\ell=1}^N \ewght{t}{\ell} L_{t,t,T}(\epart{t}{\ell},\1)} \\ \nonumber
&\hspace{130pt}
- \sum_{i_{s:t-1}}^N h(\epart{s}{i_s}) \swghtcIndex{s:t-1}{i} \frac{\ewght{t-1}{i_{t-1}} L_{t-1,t-1,T}(\epart{t-1}{i_{t-1}},\1)}{\sum_{\ell=1}^N \ewght{t-1}{\ell} L_{t-1,t-1,T}(\epart{t-1}{\ell},\1)} \\
\label{eq:decomposition-diff-A}
& \quad = \sum_{i_{s:t}=1}^N h(\epart{s}{i_s}) \swghtcIndex{s:t-1}{i} \left\{ \mu_{t-1}(i_{t-1},\epart{t}{i_t}) - \mu'_{t-1}(i_{t-1}) \right\} \frac{\ewght{t}{i_t} L_{t,t,T}(\epart{t}{i_t},\1)}{\sum_{\ell=1}^N \ewght{t}{\ell} L_{t,t,T}(\epart{t}{\ell},\1)} \eqsp,
\end{align}
where $\mu_{t-1}(\cdot,\epart{}{})$ and $\mu'_{t-1}(\cdot)$ are two probability distributions on the set $\{1,\dots,N\}$ defined as
\begin{equation}
\label{eq:definition-alpha-beta}
\mu_{t-1}(i,\epart{}{})= \frac{\ewght{t-1}{i} q(\epart{t-1}{i},\epart{}{})}{\sum_{\ell=1}^N \ewght{t-1}{\ell} q(\epart{t-1}{\ell},\epart{}{})} \quad \text{and} \quad
\mu'_{t-1}(i)= \frac{\ewght{t-1}{i} L_{t-1,t-1,T}(\epart{t-1}{i},\1)}{\sum_{\ell=1}^N \ewght{t-1}{\ell} L_{t-1,t-1,T}(\epart{t-1}{\ell},\1)} \eqsp.
\end{equation}
It follows from Lemma \ref{lem:uniform-forgetting-backward-kernel} that, for all $\epart{}{}$
\[
\left| \sum_{i_{s:t-1}=1}^N h(\epart{s}{i_s}) \swghtcIndex{s:t-1}{i} \left\{ \mu_{t-1}(i_{t-1},\epart{}{}) - \mu'_{t-1}(i_{t-1}) \right\} \right| \leq \oscnorm{h} \rho^{t-1-s} \eqsp,
\]
which in turn implies that
\begin{equation}
\label{eq:boundabcd}
\left| \frac{a_N}{b_N} - \frac{c_N}{d_N}\right| \leq \oscnorm{h} \rho^{t-1-s} \eqsp,
\end{equation}
showing condition \eqref{item:inegEssentielle-borne} with $M= \oscnorm{h} \rho^{t-s-1}$. We now consider the condition \eqref{item:inegEssentielle-expo1}. It follows from the definition of $b_N$ that:
\[
b_N=  N^{-1}\sum_{i_t=1}^N \ewght{t}{i_t} \frac{L_{t,t,T}(\epart{t}{i_t},\1)}{\esssup{L_{t,t,T}(\cdot,\1)}} \eqsp.
\]
By \eqref{eq:TU-Hoeffding-Filtering-unnormalized} and \eqref{eq:majorationL},
\[
\PP\left[| b_N-b | \geq \epsilon \right] \leq B \rme^{-C N\epsilon^2}
\]
where the constants $B$ and $C$ do not depend on the time indexes $t$ and $T$.
This relation shows condition \eqref{item:inegEssentielle-expo1}. We finally consider condition \eqref{item:inegEssentielle-expo2}. Using \eqref{eq:decomposition-diff-A}, $a_N - \frac{c_N}{d_N} b_N= \sumwght{t}^{-1} \sum_{i=1}^N \ewght{t}{i} G_t(\epart{t}{i})$,
where
\begin{multline*}
G_t(\epart{}{})=  \sum_{i_{s:t-1}=1}^N h(\epart{s}{i_s}) \swghtcIndex{s:t-1}{i} \\ \times
\left\{ \frac{\ewght{t-1}{i_{t-1}} q(\epart{t-1}{i_{t-1}},\epart{}{})}{\sum_{\ell=1}^N \ewght{t-1}{\ell}q(\epart{t-1}{\ell},\epart{}{})} - \frac{\ewght{t-1}{i_{t-1}} L_{t-1,t-1,T}(\epart{t-1}{i_{t-1}},\1)}{\sum_{\ell=1}^N \ewght{t-1}{\ell} L_{t-1,t-1,T}(\epart{t-1}{\ell},\1)} \right\} \frac{L_{t,t,T}(\epart{}{},\1)}{\esssup{L_{t,t,T}(\cdot,\1)}} \eqsp.
\end{multline*}
Using that
\begin{multline*}
\CPE{\frac{\ewght{t-1}{i_{t-1}} q(\epart{t-1}{i_{t-1}},\epart{t}{1}) \ewght{t}{1} L_{t,t,T}(\epart{t}{1},\1)}{\sum_{\ell=1}^N \ewght{t-1}{\ell} q(\epart{t-1}{\ell},\epart{t}{1})}}{\mcf{t-1}} \\
= \frac{\ewght{t-1}{i_{t-1}} \int q(\epart{t-1}{i_{t-1}},x) g_t(x) L_{t,t,T}(x,\1)\rmd x}{\sum_{\ell=1}^N \ewght{t-1}{\ell} \adjfunc{t}{t}{\epart{t-1}{\ell}}}= \frac{\ewght{t-1}{i_{t-1}} L_{t-1,t-1,T}(\epart{t-1}{i_{t-1}},\1)}{\sum_{\ell=1}^N \ewght{t-1}{\ell} \adjfunc{t}{t}{\epart{t-1}{\ell}}}
\end{multline*}
and
\[
\CPE{\ewght{t}{1} L_{t,t,T}(\epart{t}{1},\1)}{\mcf{t-1}}= \frac{\sum_{\ell=1}^N \ewght{t-1}{\ell} L_{t-1,t-1,T}(\epart{t-1}{\ell},\1)}{\sum_{\ell=1}^N \ewght{t-1}{\ell} \adjfunc{t}{t}{\epart{t-1}{\ell}}}
\]
it follows that $\CPE{\ewght{t}{1} G_t(\epart{t}{1})}{\mcf{t-1}}= 0$. On the other hand, using Lemma \ref{lem:uniform-forgetting-backward-kernel} (with $\mu$ and $\mu'$ defined in \eqref{eq:definition-alpha-beta}) and \eqref{eq:majorationL}, $|G_t(\epart{}{})| \leq \rho^{t-s-1} \oscnorm{h}$. We may therefore apply the Hoeffding
inequality to show that
\[
\PP\left[ \left| N^{-1} \sum_{i=1}^N \ewght{t}{i} G_t(\epart{t}{i}) \right| \geq \epsilon \right] \leq 2 \exp\left(-\frac{2 N \epsilon^2}{\oscnorm[2]{h} \rho^{2(t-s-1)}} \right)
\]
showing that condition \eqref{item:inegEssentielle-expo2} is satisfied with constants that do not depend on $t$. Combining these results, Lemma \ref{lem:inegEssentielle} shows that, there
exists some constants $B$ and $C$, such that, for all $s < t$,
\[
\PP\left[ \left| \frac{A_{s,t,T}(h)}{A_{s,t,T}(\1)} - \frac{A_{s,t-1,T}(h)}{A_{s,t-1,T}(\1)}\right|  \geq \epsilon \right] \leq B \exp\left(-C N \frac{\epsilon^2}{\rho^{2(t-s-1)} \oscnorm[2]{h}}\right) \eqsp.
\]
The proof is concluded by applying Lemma~\ref{lem:pasfor}.
\end{proof}

\section{A limiting expression of the variance of the marginal smoothing distribution}
\label{sec:TimeUniformCLTFFBS}
In this section, we study  the expression of the variance \eqref{eq:expression-covariance} for the FFBS approximation of the marginal smoothing distribution. In particular, we show that under the strong mixing condition (A\ref{assum:strong-mixing-condition}) the asymptotic variance
of the marginal smoothing estimator $\asymVar[\Xinit]{s:T}{T}{\bar{h}}$, where $\bar{h}$ is defined in \eqref{eq:definition-barh} has a finite limiting value has $T \to \infty$ for a given value of $s$. We will also show that this variance is upper bounded uniformly in time, allowing to construct uniform confidence intervals.

The first step consists in showing that the asymptotic variance of the auxiliary particle filter has a finite limiting value as $T\to \infty$, and deriving an upper-bound for this limit.

\begin{prop}\label{prop:uniform-bound-filtering}
Assume (A\ref{assum:bound-likelihood}--\ref{assum:strong-mixing-condition}) hold for $T=\infty$. Then, with the notations of Proposition \ref{prop:CLT-auxiliary-particle-filter},
\[ \asymVar[\Xinit]{s}{s}{h} \leq \frac{\sigma_+}{c_- \sigma_-} \sup_{r \geq 0} \esssup{\ewghtfunc{r}} \esssup{\adjfunc{r}{r}{}}  \frac{1}{1-\rho^2} \oscnorm[2]{h} \eqsp . \]
\end{prop}
\begin{proof}
Without loss of generality, we assume that $\post[\Xinit]{s}{s}(h)=0$.
Note that, for any $r \in \{0, \dots, s\}$, under (A\ref{assum:strong-mixing-condition}),
\begin{align}
\label{eq:borne1}
&\frac{\int Q(x,\rmd x') g_r(x') L_{r,s}(x',\1)}{\post[\Xinit]{r}{r-1}[g_r(\cdot) L_{r,s}(\cdot,\1)]} \leq \frac{\sigma_+}{\sigma_-} \eqsp, \\
\label{eq:borne2}
& \left|\frac{L_{r,s}(x,h)}{L_{r,s}(x,\1)}-\frac{\post[\Xinit]{r}{r}[L_{r,s}(\cdot,h)]}{\post[\Xinit]{r}{r}[L_{r,s}(\cdot,\1)]}\right|
\leq \rho^{s-r} \eqsp.
\end{align}
We now bound  $V_{\Xinit,r,s}[h]$, $r=0\dots s$, defined \eqref{eq:defVOs} and \eqref{eq:defVrs}. First, for $r=0$,
using that $\Xinit(L_{0,s}(\cdot, h))=0$, Proposition~\ref{prop:forgetting-initial-final-conditions} and inequality \eqref{eq:lower-upper-bound} show that
\begin{align*}
V_{\Xinit,0,s}[h] &= \frac{\rho_0\left( \ewghtfunc{0}^2(\cdot) L^2_{0,s}(\cdot,h) \right)}{\left(\post[\Xinit]{0}{-1}\left[g_0(\cdot) L_{0,s}(\cdot,\1) \right] \right)^2} = \frac{\Xinit\left( \frac{\rmd \Xinit}{\rmd \rho_0}(\cdot) \left[ g_0(\cdot)L_{0,s}(\cdot,h) \right]^2\right)}{\left(\post[\Xinit]{0}{-1}\left[g_0(\cdot) L_{0,s}(\cdot,\1) \right] \right)^2}\\
& = \Xinit\left[\frac{\rmd \Xinit}{\rmd \rho_0}(\cdot)\left\{\frac{g_0(\cdot)L_{0,s}(\cdot, \1)}{\Xinit\left[g_0(\cdot)L_{0,s}(\cdot, \1)\right]}\left[\frac{L_{0,s}(\cdot, h)}{L_{0,s}(\cdot, \1)}-\frac{\Xinit(L_{0,s}(\cdot, h))}{\Xinit(L_{0,s}(\cdot, \1))}\right]\right\}^2\right]\\
&\leq \frac{\supnorm{\ewghtfunc{0}{}}}{\Xinit(g_0)} \frac{\sigma_+}{\sigma_-} \rho^{2s} \oscnorm[2]{h} \leq
\frac{\supnorm{\ewghtfunc{0}{}}}{c_-} \frac{\sigma_+}{\sigma_-} \rho^{2s} \oscnorm[2]{h}\eqsp.
\end{align*}
Similarly, for $r>0$, using that $\post[\Xinit]{r}{r}[L_{r,s}(\cdot,h)]=0$, Eqs.~\eqref{eq:borne1} and \eqref{eq:borne2} show that
\begin{align*}
&\frac{\post[\Xinit]{r-1}{r-1}\left[ \adjfunc{r}{r}{\cdot} \int \kiss{r}{r}(\cdot,x) \ewghtfunc{r}^2(\cdot,x) L^2_{r,s}(x,h) \rmd x\right] }{\left(\post[\Xinit]{r}{r-1}\left[g_r(\cdot) L_{r,s}(\cdot,\1) \right] \right)^2}\\
&= \post[\Xinit]{r-1}{r-1}\left[  \int  Q(\cdot,\rmd x) g_r(x) \ewghtfunc{r}(\cdot,x) \times \dots\phantom{\frac{L_{r,s}(x,\1)}{\post[\Xinit]{r}{r-1}[g_r(\cdot) L_{r,s}(\cdot,\1)]} }\right. \\
&\hspace{60pt}\left. \left\{ \frac{L_{r,s}(x,\1)}{\post[\Xinit]{r}{r-1}[g_r(\cdot) L_{r,s}(\cdot,\1)]}  \left[ \frac{L_{r,s}(x,h)}{L_{r,s}(x,\1)}-\frac{\post[\Xinit]{r}{r}[L_{r,s}(\cdot,h)]}{\post[\Xinit]{r}{r}[L_{r,s}(\cdot,\1)]} \right]\right\}^2 \right] \\
& \leq  \frac{\esssup{\ewghtfunc{r}}}{\post[\Xinit]{r}{r-1}[g_r(\cdot)]} \frac{\sigma_+}{\sigma_-}  \rho^{2(r-s)} \oscnorm[2]{h} \leq
\frac{\esssup{\ewghtfunc{r}}}{c_-} \frac{\sigma_+}{\sigma_-}  \rho^{2(r-s)} \oscnorm[2]{h}
\end{align*}
which implies that
$V_{\Xinit,r,s}[h] \leq \esssup{\adjfunc{r}{r}{}} (\sigma_+/c_-\;\sigma_-) \rho^{2(r-s)} \oscnorm[2]{h} \esssup{\ewghtfunc{r}}$. The result follows.
\end{proof}

We are now in position to state and prove the main result of this section, which provides a uniform bound for the variance
of the particle estimator of the marginal smoothing distribution.
\begin{thm}
Assume (A\ref{assum:bound-likelihood}--\ref{assum:strong-mixing-condition}) hold for $T=\infty$. Then, for any $s \leq T$ ,
\[
\asymVar[\Xinit]{s}{T}{\bar{h}} \leq \frac{\oscnorm[2]{h}}{1-\rho^2} \left(\frac{1}{c_-}\left(\frac{\sigma_+}{\sigma_-}\right)^2 \sup_{r \geq 0} \esssup{\ewghtfunc{r}} \esssup{\adjfunc{r}{r}{}}  + \frac{\sigma_+^4}{c_-^2\sigma_-^3}
\sup_{r\geq 0}\esssup{\ewghtfunc{r}} \esssup{\adjfunc{r}{r}{}} \esssup{g_r}\right) \eqsp ,
\]
where the function $\bar{h}$ and the covariance $\asymVar[\Xinit]{s}{T}{\bar{h}}$ are defined in  \eqref{eq:definition-barh} and \eqref{eq:expression-covariance},
respectively.
\end{thm}
\begin{proof}
We bound the summands appearing in \eqref{eq:expression-covariance}.
Consider first the term $\asymVar[\Xinit]{s}{s}{L_{s,s,T}(\cdot,h)} / \post[\Xinit]{s}{s}^2[L_{s,s,T}(\cdot,\1)]$. We first apply Proposition \ref{prop:uniform-bound-filtering} to the function $L_{s,s,T}(\cdot, \bar{h})/\post[\Xinit]{s}{s}\left[L_{s,s,T}(\cdot, \1)\right]$.
Using \eqref{eq:lower-upper-bound} and $L_{s,s,T}(\cdot, \bar{h})= h(\cdot) L_{s,s,T}(\cdot,\1)$, $$ \frac{1}{\int \post[\Xinit]{s}{s}(\rmd x_s)\left[\frac{L_{s,s,T}(x_s, \1)}{L_{s,s,T}(x, \1)}\right]}\leq\frac{\sigma_+}{\sigma_-}, $$
hence $\oscnorm{ L_{s,s,T}(\cdot, \bar{h})/\post[\Xinit]{s}{s}\left[L_{s,s,T}(\cdot, \1)\right]} \leq \frac{\sigma_+}{\sigma_-} \oscnorm{h}$, and
\begin{equation}\label{eq:uniform-smoothing-term1}
\frac{\asymVar[\Xinit]{s}{s}{L_{s,s,T}(\cdot,h)}}{\post[\Xinit]{s}{s}^2[L_{s,s,T}(\cdot,\1)]}\leq \frac{1}{c_-}\left(\frac{\sigma_+}{\sigma_-}\right)^2 \sup_{r \geq 0} \esssup{\ewghtfunc{r}} \esssup{\adjfunc{r}{r}{}}  \frac{ \oscnorm[2]{h}}{1-\rho^2} \eqsp.
\end{equation}
Now, we write
\begin{align}
&\frac{\post[\Xinit]{t-1}{t-1}(\upsilon_{s,t,T}(\cdot,h))}{\post[\Xinit]{t}{t-1}^2[g_t(\cdot) L_{t,t,T}(\cdot,\1)]} \nonumber\\ %\label{eq:majunifsmooth-term}\\
 &= \post[\Xinit]{t-1}{t-1}\left( \adjfunc{t}{t}{\cdot} \int \frac{  \kiss{t}{t}(\cdot, x') \ewghtfunc{t}^2(\cdot, x') \post[\Xinit]{s:t-1}{t-1}^2\left[ \bar{h}(\cdot) q_{s,t-1}(\cdot,x')\right] L_{t,t,T}^2(x', \1) } {\post[\Xinit]{t-1}{t-1}^2\left[q(\cdot, x')\right] \post[\Xinit]{t}{t-1}^2\left[g_t(\cdot) L_{t,t,T}(\cdot, \1)\right] } \rmd x'\right)\eqsp .\nonumber
\end{align}
We will show that
\begin{equation}\label{eq:keyineg}
 \left| \frac{\post[\Xinit]{s:t-1}{t-1}\left[ \bar{h}(\cdot) q_{s,t-1}(\cdot,x')\right] L_{t,t,T}(x', \1)}{\post[\Xinit]{t}{t-1}\left[g_t(\cdot) L_{t,t,T}(\cdot, \1)\right]} \right| \leq \frac{\sigma_+^2}{\sigma_- c_-} \rho^{t-s} \oscnorm{h}\eqsp.
\end{equation}
Using this inequality,
\begin{align*}
&\frac{\post[\Xinit]{t-1}{t-1}(\upsilon_{s,t,T}(\cdot,h)) \post[\Xinit]{t-1}{t-1}(\adjfunc{t}{t}{})}{\post[\Xinit]{t}{t-1}^2[g_t(\cdot) L_{t,t,T}(\cdot,\1)]}\\
 &\leq
 \left(\frac{\sigma_+^2}{\sigma_- c_-}\right)^2 \rho^{2(t-s)} \oscnorm{h}^2
 \post[\Xinit]{t-1}{t-1}\left( \adjfunc{t}{t}{\cdot} \int \frac{  \kiss{t}{t}(\cdot, x') \ewghtfunc{t}^2(\cdot, x') }{\post[\Xinit]{t-1}{t-1}^2\left[q(\cdot, x')\right]}\rmd x'\right) \esssup{\adjfunc{t}{t}{}}\\
&\leq \left(\frac{\sigma_+^2}{\sigma_- c_-}\right)^2 \rho^{2(t-s)}  \oscnorm{h}^2
 \post[\Xinit]{t-1}{t-1}\left( \int \frac{ q(\cdot, x') g_t(x') \ewghtfunc{t}(\cdot, x') }{\post[\Xinit]{t-1}{t-1}^2\left[q(\cdot, x')\right]}\rmd x'\right)\esssup{\adjfunc{t}{t}{}}\\
&\leq \left(\frac{\sigma_+^2}{\sigma_- c_-}\right)^2 \rho^{2(t-s)} \oscnorm{h}^2
\frac{\esssup{\ewghtfunc{t}} \esssup{\adjfunc{t}{t}{}} \esssup{g_t} }{\sigma_-}\eqsp ,
\end{align*}
where we used the Fubini Theorem in the last step. Let us finally turn to the proof of the inequality \eqref{eq:keyineg}.
\begin{align*}
 &\frac{\post[\Xinit]{s:t-1}{t-1}\left[ \bar{h}(\cdot) q_{s,t-1}(\cdot,x')\right] L_{t,t,T}(x', \1)}{\post[\Xinit]{t}{t-1}\left[g_t(\cdot) L_{t,t,T}(\cdot, \1)\right]} \\
 &= \frac{\int \post[\Xinit]{s}{s}(\rmd x_s) h(x_s) l_{s,t-1}(x_s, x_{t-1})q(x_{t-1}, x') L_{t,t,T}(x', \1) \rmd x_{t-1}}{\int \post[\Xinit]{s}{s}(\rmd x_s) l_{s, t-1} (x_s, x_{t-1}) \left\{ \int g_t(x_t) q(x_{t-1}, x_t) L_{t,t,T}(x_t, \1) \rmd x_t\right\} \rmd x_{t-1}} \eqsp.
\end{align*}
The last expression can be written $A\times B$ with
$$A= \frac{\int \post[\Xinit]{s}{s}(\rmd x_s) h(x_s) l_{s,t-1}(x_s, x_{t-1})q(x_{t-1}, x') \rmd x_{t-1}}{\int \post[\Xinit]{s}{s}(\rmd x_s) l_{s, t-1} (x_s, x_{t-1}) q(x_{t-1}, x') \rmd x_{t-1}}$$
and
$$B= \frac{\int \post[\Xinit]{s}{s}(\rmd x_s) l_{s, t-1} (x_s, x_{t-1}) q(x_{t-1}, x') L_{t,t,T}(x', \1)  \rmd x_{t-1}} {\int \post[\Xinit]{s}{s}(\rmd x_s) l_{s, t-1} (x_s, x_{t-1}) \left\{ \int g_t(x_t) q(x_{t-1}, x_t) L_{t,t,T}(x_t, \1) \rmd x_t\right\} \rmd x_{t-1}}\eqsp.$$
We will bound these two terms separately.
Since $\post[\Xinit]{s}{T}(h) = 0$,
$$A' = \frac{\int \post[\Xinit]{s}{s}(\rmd x_s) h(x_s) l_{s,t-1}(x_s, x_{t-1}) L_{t-1,t-1, T}(x_{t-1}, \1) \rmd x_{t-1}}{\int \post[\Xinit]{s}{s}(\rmd x_s) l_{s,t-1}(x_s, x_{t-1}) L_{t-1, t-1, T}(x_{t-1}, \1) \rmd x_{t-1}} = 0 \eqsp.$$
Thus, by Proposition~\ref{prop:forgetting-initial-final-conditions},
\begin{align*}
|A| &= \left|A-A'\right|\leq \rho^{t-s} \oscnorm{h}\eqsp.
\end{align*}
On the other hand, as $q(x_{t-1}, x')\leq \sigma_+$, as $\int g_t(x_t) q(x_{t-1}, x_t)\rmd x_t \geq c_-$
and as for every $x_t$ it holds that $$\frac{L_{t,t, T}(x_{t}, \1)}{L_{t,t, T}(x', \1)} \geq \frac{\sigma_-}{\sigma_+}\eqsp ,$$
$B$ is upper-bounded by $\sigma_+^2 / (\sigma_-c_-)$.
\end{proof}

\appendix
\section{Technical results}
\begin{lem}\label{lem:inegEssentielle}
Assume that $a_N$, $b_N$, $c_N$, $d_N$ and $b$ are random variables such that there exist positive constants $\beta,B_1,C_1,B_2,C_2,M$ such that
\begin{enumerate}[(I)]
\item \label{item:inegEssentielle-borne}
$|a_N/b_N - c_N/d_N|\leq M$, $\PP$-\as\ and  $b \geq \beta$, $\PP$-\as
\item \label{item:inegEssentielle-expo1}
For all $\epsilon>0$ and all $N\geq1$, $\CPP{|b_N-b|>\epsilon}{}{}\leq B_1 \rme^{-C_1 N \epsilon^2}$ \eqsp,
\item \label{item:inegEssentielle-expo2}
For all $\epsilon>0$ and all $N\geq1$, $\PP(|a_N-(c_N/d_N)b_N|>\epsilon)\leq B_2 \rme^{-C_2 N \left(\frac{\epsilon}{M}\right)^2}$
\end{enumerate}
Then,
$$
\PP\left( \left|\frac{a_N}{b_N}-\frac{c_N}{d_N}\right|>\epsilon\right) \leq B_1 \rme^{-C_1 N \left(\frac{\epsilon \beta}{2M}\right)^2}+ B_2 \rme^{-C_2 N \left(\frac{\epsilon \beta}{2M}\right)^2}
$$
\end{lem}
\begin{proof}
Write
\begin{multline*}
\left|\frac{a_N}{b_N}-\frac{c_N}{d_N}\right|  \leq b^{-1} \left| \frac{a_N}{b_N} - \frac{c_N}{d_N} \right|  |b-b_N| + b^{-1}\left| a_N - \frac{c_N}{d_N} b_N \right|\\
\leq \beta^{-1}M |b-b_N|+ \beta^{-1}\left| a_N - \frac{c_N}{d_N} b_N \right| \quad a.s.
\end{multline*}
Thus,
$$
\left\{\left|\frac{a_N}{b_N}-\frac{c_N}{d_N}\right|>\epsilon\right\} \subset \left\{|b-b_N|> \frac{\epsilon \beta}{2M}\right\} \cup \left\{ \left|a_N - \frac{c_N}{d_N} b_N\right|>\frac{\epsilon \beta}{2M}\right\}
$$
and the proof follows.
\end{proof}

\begin{lem}
\label{lem:toutbete}
Let $\nu$ be a  measure and $\{A_N(x)\}$ be a sequence of stochastic processes such that,
\begin{enumerate}
\item for $\nu$-almost every $x$, $A_N(x) \plim a(x)$,
\item there exists a constant $C$ and a $\nu$-integrable function $h$, for $\nu$-almost every $x$, $|A_N(x)| \leq C h(x)$.
\end{enumerate}
Then, $\int | A_N(x) -a(x)| \nu(\rmd x) \to_{N \to \infty} 0$ in $L^1(\nu)$.
\end{lem}
\begin{proof}
Under the stated assumptions $\PE| A_N(x) - a(x)| \to 0$ as $N$ goes to infinity $\nu$-a.e.. On the other hand, $\nu$-a.e., $|A_N(x) - a(x)| \leq 2C h(x)$.
The proof follows from the Fubini Theorem and the dominated convergence Theorem,
\[
\PE  \int |A_N(x)  -a(x)| \nu(\rmd x)  =\int \PE \left| A_N(x) -a (x) \right| \nu(\rmd x) \rightarrow_{N \to \infty} 0 \eqsp.
\]
\end{proof}

\begin{lem}
\label{lem:toutbete-1}
Assume that (A\ref{assum:bound-likelihood}-\ref{assum:borne-FFBS}) hold for some $T$. Let $\{\Upsilon_N(x)\}$ be a sequence of stochastic processes and $\upsilon$
a function such that
\textbf{(i)} there exists a constant $C < \infty$ such that, for all $N$, $\esssup{\Upsilon_N} \leq \upsilon_\infty$ and $\esssup{\upsilon} \leq C$ and
\textbf{(ii)} for all $M \geq 0$, $\sup_{|x| \leq M} \left|\Upsilon_N(x) - \upsilon(x) \right| \plim_{N \to \infty} 0$. Then, $t \leq T$,
$\sumwght{t}^{-1} \sum_{\ell=1}^N \ewght{t}{\ell} \Upsilon_N(\epart{t}{\ell}) \plim_{N \to \infty} \post[\Xinit]{t}{t}(\upsilon)$.
\end{lem}
\begin{proof}
Write $\sumwght{N}^{-1} \sum_{t=1}^N \ewght{t}{\ell} \{ \Upsilon_N(\epart{t}{\ell}) - \upsilon(\epart{t}{\ell}) \} = S_{N,1} + S_{N,2}$,
with $S_{N,1}\eqdef \sumwght{N}^{-1} \sum_{t=1}^N \ewght{t}{\ell} \{ \Upsilon_N(\epart{t}{\ell}) - \upsilon(\epart{t}{\ell}) \}  \1 \{ |\epart{t}{\ell}| \leq M \}$ and $S_{N,2} \eqdef \sumwght{N}^{-1} \sum_{t=1}^N \ewght{t}{\ell} \{ \Upsilon_N(\epart{t}{\ell}) - \upsilon(\epart{t}{\ell}) \}  \1 \{ |\epart{t}{\ell}| > M \}$. Since $S_{N,1} \leq \sup_{|x| \leq M} \left|\Upsilon_N(x) - \upsilon(x)\right| $, assumption \textbf{(ii)} implies that $S_{N,1} \plim_{N \to \infty} 0$. On the other hand, $S_{N,2} \leq 2 C \sumwght{t}^{-1} \sum_{\ell=1}^N \ewght{t}{\ell} \1 \{ |\epart{t}{\ell}| > M\}$.
By Proposition \ref{prop:exponential-inequality-forward}, $\sumwght{t}^{-1} \sum_{\ell=1}^N \ewght{t}{\ell} \1 \{ |\epart{t}{\ell}| > M\} \plim_{N \to \infty} \post[\Xinit]{t}{t}\left( \1 \{ |\cdot| > M \} \right)$. the proof follows since $\lim_{M \to \infty} \post[\Xinit]{t}{t}\left( \1 \{ |\cdot| > M \} \right)= 0$.
\end{proof}

\begin{lem}
\label{lem:pasfor}
Let $\{ Y_{n,i} \}_{i=1}^n$ be a triangular array of random variables such that there exist constants $B > 0$, $C > 0$ and $\rho$, $0 < \rho < 1$ such that, for all $n$, $i \in \{1, \dots, n \}$ and $\epsilon > 0$,
\[
\PP\left( |Y_{n,i}| \geq \epsilon \right) \leq B \rme^{-C \epsilon^2 \rho^{-2i}} \eqsp.
\]
Then, there exists $\bar{B}$ and $\bar{C}$ such that, for any $n$ and $\epsilon > 0$,
\[
\PP\left( \left| \sum_{i=1}^n Y_{n,i} \right| \geq \epsilon \right) \leq \bar{B} \rme^{- \bar{C} \epsilon^2} \eqsp.
\]
\end{lem}
\begin{proof}
Denote by $S \eqdef \sum_{i=1}^\infty \sqrt{i} \rho^i$. It is plain to see that
\[
\PP\left( \left| \sum_{i=1}^n Y_{n,i} \right| \geq \epsilon \right) \leq \sum_{i=1}^n \PP\left( |Y_{n,i}| \geq \epsilon S^{-1} \sqrt{i} \rho^i \right) \leq B \sum_{i=1}^n \rme^{-C S^{-1} \epsilon^2 i} \eqsp.
\]
Set $\epsilon_0 > 0$. The proof follows by noting that, for any $\epsilon \geq  \epsilon_0$,
\[
\sum_{i=1}^n \rme^{-C S^{-1} i \epsilon^2} \leq (1- \rme^{C S^{-1} \epsilon_0^2})^{-1} \rme^{C S^{-1} \epsilon_0^2} \rme^{-C S^{-1} \epsilon^2} \eqsp.
\]
\end{proof}
\bibliographystyle{plain}
%\bibliography{../motherofallbibs/motherofallbibs}
%\bibliography{../motherofallbibs/motherofallbibs}
\bibliography{dgarm}
\end{document}